\documentclass[onefignum,onetabnum]{siamonline220329}

\usepackage{amsfonts}
\usepackage{amssymb}
\usepackage{amsmath}
\usepackage{hyperref}
\usepackage{booktabs} 
\usepackage{caption} 
\usepackage{subcaption} 
\usepackage{graphicx}
\usepackage{pgfplots}
\usepackage[all]{nowidow}
\usepackage[utf8]{inputenc}
\usepackage{tikz}
\usetikzlibrary{er,positioning,bayesnet}
\usepackage{multicol}
\usepackage{algpseudocode,algorithm,algorithmicx}

\usepackage{epstopdf}
\ifpdf
  \DeclareGraphicsExtensions{.eps,.pdf,.png,.jpg}
\else
  \DeclareGraphicsExtensions{.eps}
\fi

\newsiamremark{remark}{Remark}
\newsiamremark{hypothesis}{Hypothesis}
\crefname{hypothesis}{Hypothesis}{Hypotheses}
\newsiamthm{claim}{Claim}

\newcommand\hl[1]{#1}

\newcommand{\fa}{\mathbb{\forall}}

\newcommand{\R}{\mathbb{R}}

\newcommand{\subspace}{\hookrightarrow}

\newcommand{\C}{\mathbb{C}}

\newcommand{\cl}{\overline}

\newcommand{\cB}{\mathcal{B}}
\newcommand{\cD}{\mathcal{D}}
\newcommand{\cF}{\mathcal{F}}
\newcommand{\cM}{\mathcal{M}}
\newcommand{\cQ}{\mathcal{Q}}
\newcommand{\cR}{\mathcal{R}}
\newcommand{\cS}{\mathcal{S}}
\newcommand{\cZ}{\mathcal{Z}}

\newcommand{\bLambda}{\boldsymbol{\Lambda}}
\newcommand{\balpha}{\boldsymbol{\alpha}}
\newcommand{\bbeta}{\boldsymbol{\beta}}
\newcommand{\bzero}{\boldsymbol{0}}

\newcommand{\bb}{\mathbf{b}}
\newcommand{\bc}{\mathbf{c}}
\newcommand{\bd}{\mathbf{d}}
\newcommand{\bff}{\mathbf{f}}
\newcommand{\bm}{\mathbf{m}}
\newcommand{\bn}{\mathbf{n}}
\newcommand{\bq}{\mathbf{q}}
\newcommand{\br}{\mathbf{r}}
\newcommand{\bs}{\mathbf{s}}

\newcommand{\bu}{\mathbf{u}}
\newcommand{\bv}{\mathbf{v}}
\newcommand{\bw}{\mathbf{w}}
\newcommand{\bx}{\mathbf{x}}
\newcommand{\by}{\mathbf{y}}
\newcommand{\bz}{\mathbf{z}}

\newcommand{\tbd}{\widetilde{\bd}}

\newcommand{\hq}{\widehat{q}}

\newcommand{\bB}{\mathbf{B}}
\newcommand{\bE}{\mathbf{E}}
\newcommand{\bI}{\mathbf{I}}
\newcommand{\bJ}{\mathbf{J}}
\newcommand{\bM}{\mathbf{M}}
\newcommand{\bQ}{\mathbf{Q}}
\newcommand{\bP}{\mathbf{P}}
\newcommand{\bS}{\mathbf{S}}
\newcommand{\bT}{\mathbf{T}}
\newcommand{\bX}{\mathbf{X}}

\newcommand{\bZ}{\mathbf{Z}}

\headers{Inverse scattering for Schr\"{o}dinger equation via ROM}{A. Tataris, T. van Leeuwen, and A.V. Mamonov}

\title{Inverse scattering for Schr\"{o}dinger equation in the frequency domain via 
data-driven reduced order modeling
\thanks{Submitted to SIIMS.
\funding{A.V. Mamonov was supported in part by the U.S. National Science
Foundation under award DMS-2309197. This material is based upon research supported in part
by the U.S. Office of Naval Research under award number N00014-21-1-2370 to A.V. Mamonov.}}}

\author{Andreas Tataris\thanks{Delft University of Technology, The Netherlands 
  (\email{A.Tataris@tudelft.nl}).}
\and Tristan van Leeuwen\thanks{Utrecht University and CWI Amsterdam, The Netherlands 
  (\email{t.vanleeuwen@uu.nl}).} 
\and Alexander V. Mamonov\thanks{Department of Mathematics, University of Houston, Houston, TX, USA 
  (\email{avmamonov@uh.edu}).}}

\usepackage{amsopn}

\ifpdf
\hypersetup{
  pdftitle={Inverse scattering for Schr\"{o}dinger equation via ROM}{A. Tataris, T. van Leeuwen, and A.V. Mamonov},
  pdfauthor={A. Tataris, T. van Leeuwen, and A.V. Mamonov}
}
\fi

\begin{document}

\maketitle

\begin{abstract}
In this paper we develop a numerical method for solving an inverse scattering problem of estimating 
the scattering potential in a Schr\"{o}dinger equation from frequency domain measurements based on 
reduced order models (ROM). The ROM is a projection of Schr\"{o}dinger operator onto a subspace spanned 
by its solution snapshots at certain wavenumbers. Provided the measurements are performed at these 
wavenumbers, the ROM can be constructed in a data-driven manner from the measurements on a surface 
surrounding the scatterers. Once the ROM is computed, the scattering potential can be estimated using 
non-linear optimization that minimizes the ROM misfit. Such an approach typically outperforms the conventional
methods based on data misfit minimization. We develop two variants of ROM-based algorithms for 
inverse scattering and test them on a synthetic example in two spatial dimensions.
\end{abstract}



\section{Introduction}

Inverse scattering problems are ubiquitous in the areas of science and engineering where one wants to 
reconstruct either non-penetrable scatterers or to examine material properties of media without direct 
physical access to it. In this paper we focus on an inverse scattering problem that belongs to the latter 
type, of estimating the scattering potential of a Schr\"{o}dinger operator in two or more dimensions. 
Such inverse problem can be treated in the unbounded case using elegant analytical tools leading to the 
Gelfand-Levitan and Marchenko equation \cite{Cheney84}. However, in this paper we approach the solution 
of the inverse problem by combining nonlinear optimization with the techniques originating from the so-called 
reduced order models.

Reduced order models (ROM) have been used extensively in numerical analysis of PDEs as alternative to 
conventional solution methods due to their fast convergence properties and cmomputational efficiency, e.g., see 
\cite{Maday2002APC,Veroy2003APE}. In particular, projection ROMs are used to compute approximate 
solutions of PDEs as expansions in the bases of solutions to the same PDE corresponding to various 
values of the scalar parameter of the problem, e.g., time or wavenumber.
In addition, over the last several years data-driven projection ROMs were employed to construct efficient
numerical methods for solving a variety of inverse problems. These studies include inversion for coefficients
of diffusive PDEs from frequency-domain measurements, e.g., see 
\cite{borcea2020spectraldomain,borcea2014model,druskin2021lippmann,druskin2022extension},
as well as estimating coefficients of wave PDEs from time-domain data 
\cite{borcea2020reduced,borcea2021reduced,borcea2023waveform1,borcea2023waveform2,borcea2023data,
druskin2016direct,druskin2018nonlinear,mamonov2022velocity}.
For a while a question stood as to whether data-driven ROM-based approaches are applicable to
wave PDEs in the frequency domain. Recently, this question was resolved positively in 
\cite{tataris2023reduced,TristanAndreas23ROM}, where data-driven ROM techniques were successfully 
applied to the numerical solution of a classical inverse scattering problem in one spatial dimension as an 
alternative to the conventional Gelfand-Levitan-Marchenko point of view. In this work we extend this 
approach to solve numerically an inverse scattering problem for the scattering potential of Schr\"{o}dinger 
equation in two or more spatial dimensions in a bounded domain from the knowledge of frequency-domain
measurements on the boundary of the domain of interest. 

Most of the ROM-based approaches to solving inverse problems discussed above share a similar structure. 
Typically, such inversion procedures consist of two stages. In the first stage one computes from the measured
data a ROM of the PDE operator, hence the designation data-driven. The ROM is a projection of the PDE 
operator with the unknown coefficient on a subspace spanned by the solutions of the PDE for a \hl{number of 
values of} the scalar parameter, time of wavenumber depending on the measurement setting. Even though 
the projection involves an unknown PDE operator and the solutions in the bulk that are also unknown, 
certain algebraic techniques as well as the properties of PDE operators make it possible to compute such
projections from the knowledge of the data only, typically measured on or near the boundary of the domain 
of interest. In the second stage one needs to extract the information about the PDE coefficient of interest
from the data-driven ROM computed in the first stage. One option is to use nonlinear optimization to 
minimize the least squares misfit between the data-driven ROM and the reduced model computed for
a trial coefficient of the PDE. While each iteration of such optimization is as expensive as minimizing 
the least squares data misfit, a conventional approach to solving inverse problems, minimizing ROM
misfit is almost always superior to data misfit minimization. In particular, for wave problems ROM
misfit optimization objective is much better behaved (close to convex) \hl{than} data misfit objective, as shown 
in extensive numerical studies in \cite{borcea2023waveform1,borcea2023data}. Thus, ROM-based 
optimization converges much faster and is less prone to getting stuck in local minima, which leads
also to decreased sensitivity to the initial guess and ultimately to higher quality estimates of the PDE
coefficient.

We follow here the outline for ROM-based inversion described above. First, from the knowledge of 
boundary data we reconstruct the Galerkin projection of the continuous Schrö- dinger operator
onto the space spanned by the solutions of the Schrödinger equation corresponding to several 
wave numbers. This is made possible by the formulas for data-driven computation of mass and stiffness 
matrices for the Schrödinger equation in Galerkin framework. Next, the unknown Schrödinger
coefficient is estimated by solving a nonlinear optimization problem that minimizes the ROM misfit.
We explore two possible formulations, that are the two main contributions in this work. 
One minimizes the stiffness matrix misfit directly as in
\cite{TristanAndreas23ROM}, while the other transforms the stiffness matrix to a block-tridiagonal 
form using block-Lanczos process prior to minimization, as motivated by \cite{borcea2020spectraldomain,borcea2014model,tataris2023reduced}. 

The paper is organized as follows. We begin in Section~\ref{sec:prelim} by setting up the forward 
problem, as well as data model and the particular formulation of the inverse scattering problem. 
We also introduce projection ROM of the kind that is computable in a data-driven way. 
Section~\ref{sec:mainres} contains the main theoretical results as well as numerical algorithms for 
computing the ROM from the data, transforming ROM matrices to block tridiagonal form and setting 
the nonlinear optimization problem of minimizing ROM misfit to estimate the scattering potential. 
We continue in Section~\ref{sec:num} presenting the details of implementation of the proposed 
approach and the results of numerical experiments in two spatial dimensions. We conclude in
Section~\ref{sec:conclude} with a brief discussion of the results and directions for future research.

\section{Preliminaries and problem formulation}
\label{sec:prelim}

\subsection{Forward model}
\label{sec:forward}

In classical inverse scattering one considers the Schr\"{o}dinger equation
\begin{equation}
\left[ - \Delta + q(\bx) - k^2 \right] u(\bx; k) = 0, \quad \bx = [x_1, \ldots, x_d] \in \R^d,
\label{eqn:schrodrd}
\end{equation}
in the whole space $\R^d$ for $d=2,3$. The total wavefield is decomposed 
into an incoming wave $u^{\mbox{\scriptsize inc}}(\bx; k)$ and scattered wave 
$u^{\mbox{\scriptsize scat}}(\bx; k)$ so that 
\begin{equation}
u(\bx; k) = u^{\mbox{\scriptsize inc}}(\bx; k) + u^{\mbox{\scriptsize scat}}(\bx; k)
\end{equation}
and the scattered wavefield satisfies a radiation condition at infinity, e.g., Sommerfeld condition
\begin{equation}
\lim_{r \to \infty} r^{\frac{d-1}{2}} \left( \frac{\partial}{\partial r} - \imath k \right) 
u^{\mbox{\scriptsize scat}}(\bx; k) = 0,
\label{eqn:sommer}
\end{equation}
where $r = \sqrt{x_1^2 + \ldots + x_d^2}$. Here $k \in \R_+$ is the wavenumber and the scattering
potential $q(\bx)$ is the quantity of interest. Typically, it is assumed that the target scatterer  
\hl{has compact support } 
in a bounded domain $\Omega \subset \R^d$, i.e., $\overline{\mbox{supp}(q) }\subset \subset \Omega$.
Then, the target scatterer is illuminated by a number of incident wavefields and the corresponding 
far-field scattered wavefields are measured. The classical inverse scattering problem is then to recover 
$q$ 
from these measurements, see for example \cite{newton1989inverse,NewtonIII1980,cheney1984Inverse}. 

For a more realistic measurement setting we modify the problem \eqref{eqn:schrodrd}--\eqref{eqn:sommer}
to work with a finite domain $\Omega \subset \R^d$ instead of the whole $\R^d$. Thus, we pose
the Schrödinger equation in $\Omega$ only:
\begin{equation}
\left[ - \Delta + q(\bx) - k^2 \right] u^{(s)}(\bx; k) = 0, \quad \bx \in \Omega,
\label{eqn:schrodom}
\end{equation}
More importantly, we approximate the radiation condition \eqref{eqn:sommer} with an 
impedance boundary condition
\begin{equation}
\left[ \bn(\bx) \cdot \nabla - \imath k \right] u^{(s)}(\bx; k) = p_s(\bx), \quad \bx \in \partial \Omega, 
\label{eqn:bc}
\end{equation}
where $\bn: \partial \Omega \to \R^d$ is the outward facing normal on $\partial \Omega$ and the dot
denotes the standard inner product of vectors in $\R^d$. Right away we point a few differences between 
the conventional formulation \eqref{eqn:schrodrd}--\eqref{eqn:sommer} and our setting 
\eqref{eqn:schrodom}--\eqref{eqn:bc}. First, the total wavefield is no longer decomposed into an
incoming and scattered components. This leads to the second difference, the non-zero right-hand side of
the boundary condition \eqref{eqn:bc}. Note that the zero right hand side in \eqref{eqn:bc} turns it into
the absorbing boundary condition introduced in \cite{engquist1977absorbing}. 
In our setting the term 
\begin{equation}
p_s \in {H^{1/2}(\partial \Omega; \mathbb{R})}, \quad s = 1,\ldots,m,
\label{eqn:src}
\end{equation}
now plays the role of the source of illumination.
We refer to $p_s$ as the sources, and to facilitate illumination of the target scatterer from 
multiple directions, we employ $m$ of them assuming they satisfy a non-overlapping condition
\begin{equation}
\mbox{supp}(p_{s_1}) \cap \mbox{supp}(p_{s_2}) = \varnothing, \mbox{ for } s_1 \neq s_2.
\label{eqn:srcsupp}
\end{equation}
The wavefield solutions of \eqref{eqn:schrodom}--\eqref{eqn:bc} corresponding to the excitation by 
source $s$ are denoted by $u^{(s)}(\bx; k)$, $s = 1, \ldots, m$. 

Note that we can relate the formulation \eqref{eqn:schrodom}--\eqref{eqn:bc} with boundary sources
to that with incoming waves by choosing
\begin{equation}
p_s(\bx) = \xi^{(s)}(\bx) \left[ \bn(\bx) \cdot \nabla - \imath k \right] 
u^{\mbox{\scriptsize inc}, (s)}(\bx; k),
\quad \bx \in \partial\Omega, \quad s = 1,\ldots,m,
\label{eqn:srcinc}
\end{equation}
where we select $m$ incoming waves $u^{\mbox{\scriptsize inc}, (s)}(\bx; k)$ in such way that 
the right-hand side of \eqref{eqn:srcinc} does not depend on the wavenumber $k$.
The factors $\xi^{(s)}(\bx)$ are the indicator functions of $\mbox{supp}(p_{s})$ chosen to
satisfy \eqref{eqn:srcsupp}. Physically, they represent windows \hl{through} which the incoming wave 
passes before illuminating the domain of interest $\Omega$.

Assuming that the target scatterer satisfies $q \in L^\infty_+(\Omega):=L^\infty(\Omega; (0,\infty))$,
the forward problem \eqref{eqn:schrodom}--\eqref{eqn:bc} admits a weak (variational) formulation 
\begin{equation}
\int_\Omega \overline{\nabla u^{(s)}} \cdot \nabla {v} \; d \bx + 
\int_{\Omega} q \overline{u^{(s)}} v d \bx - 
k^2 \int_\Omega \overline{u^{(s)}} v d \bx + 
\imath k \int_{\partial \Omega} \overline{u^{(s)}} v d\Sigma = 
\int_{\partial \Omega} \overline{p_s} v d\Sigma, 
\label{eqn:weaksom}
\end{equation}
for all $v \in H^1(\Omega)$ and $s = 1,\ldots,m$. The existence and uniqueness of solutions of
\eqref{eqn:weaksom} is guaranteed by the following result.

\begin{theorem}
Given \hl{$k \in \R_+:=(0,\infty)$},  $q\in L^\infty_+(\Omega)$ with \hl{$\overline{\text{supp(q)}}\subset \subset \Omega $}, \hl{and boundary source} $p_s \in {H^{1/2}(\partial \Omega)}$,
the problem \eqref{eqn:weaksom} admits the unique weak solution 
$u^{(s)}(\;\cdot\; ; k) \in H^1(\Omega)$. 
\label{thm:fwdpde}
\end{theorem}
The proof is similar to that for the Helmholtz case \cite{Wald2018,tataris2023reduced}. 
\hl{For a sketch of the proof we refer to Appendix} \ref{app:fwdpde}.
Later, we will assume that the source functions admit real values.
\subsection{Operator form and Fr\'{e}chet differentiation}

In this section we introduce the operator formulation of the forward model. Its use is twofold.
First, the ROM framework developed in this work relies on finite-dimensional projections of the 
operators entering this formulation, as explained in Section~\ref{sec:projrom}. Second, the operator 
formulation helps in deriving the Fr\'{e}chet derivative of the solution wavefield with respect 
to the wave number $k$. These derivatives are required for computing the ROM from the measurements.

We begin by defining the three operators
\begin{equation}
\cS, \cM, \cB : H^1(\Omega) \to {H^1(\Omega)}'.
\label{eqn:op3}
\end{equation}
Let $\langle \cdot, \cdot \rangle_{H^1(\Omega)^\prime}$ be the pairing between $H^1(\Omega)$ and the 
corresponding space of distributions ${H^1(\Omega)}'$. Then, for all $f, g \in H^1(\Omega)$ 
we define $\cS, \cM$ and $\cB$ by
\begin{align}
\left\langle \cS f, g \right\rangle_{H^1(\Omega)^\prime}
& = \int_\Omega \overline{\nabla f} \cdot \nabla g d \bx + \int_{\Omega} q \overline{f} g d \bx, 
\label{eqn:ops}\\
\left\langle \cM f, g \right\rangle_{H^1(\Omega)^\prime} 
& = \int_{\Omega} \overline{f} g d \bx, 
\label{eqn:opm}\\
\left\langle \cB f, g \right\rangle_{H^1(\Omega)^\prime} 
& = \int_{\partial \Omega} \overline{f} g d\Sigma.
\label{eqn:opb}
\end{align}
Next, we introduce
\begin{equation}
F^{(s)} : (k, u) \in \R_+ \times H^1(\Omega) \to {H^1(\Omega)^\prime}, \quad s=1,\ldots,m,
\end{equation}
as
\begin{equation}
F^{(s)}(k, u) = (\cS - k^2 \cM + \imath k \cB) u - P^{(s)}, 
\quad (k, u) \in \R_+ \times H^1(\Omega), 
\quad s=1,\ldots,m,
\label{eqn:opf}
\end{equation}
where $P^{(s)} \in H^1(\Omega)^\prime$ is defined by
$$
\left\langle P^{(s)}, g \right\rangle_{H^1(\Omega)^\prime} 
=\int_{\partial \Omega}\overline{p_s} g d\Sigma = 
\int_{\partial \Omega}p_s g d\Sigma
$$
for any $g \in H^1(\Omega)$, assuming that $p_s \in H^{1/2}(\partial \Omega;\R), \ s=1,...,m. $

\begin{remark}
The weak problem \eqref{eqn:weaksom} can be expressed in the operator form using 
\eqref{eqn:ops}--\eqref{eqn:opf} as a problem of finding $u^{(s)}(\cdot ; k) \in H^1(\Omega)$ satisfying
\begin{equation}
(\cS - k^2 \cM + \imath k \cB) u^{(s)}(\cdot ; k) = P^{(s)},
\quad s = 1,\ldots,m. 
\label{eqn:opform}
\end{equation}
\end{remark}

Operator formulation \eqref{eqn:opform} allows to differentiate the solution $u^{(s)}(\bx;k)$ 
in Fr\'{e}chet sense with respect to the wavenumber $k$ by applying the implicit function theorem 
to $F^{(s)}$. The following result establishes the boundary value problem satisfied by Fr\'{e}chet derivative.  

\begin{theorem}
\label{thm:uderiv}
Let $q\in L^\infty_+(\Omega)$. The wavefields $u^{(s)}(\bx ; k)$ satisfying 
\eqref{eqn:schrodom}--\eqref{eqn:bc} are differentiable with respect to $k$ and their Fr\'{e}chet 
derivatives at $k=k_0$ denoted by $\partial_k u^{(s)}(\bx; k_0) = w^{(s)}(\bx)$, are weak solutions of the following problem
\begin{equation}
\left[ -\Delta + q(\bx) - k_0^2 \right] w^{(s)}(\bx) = 2 k_0 u^{(s)}(\bx; k_0), 
\quad \bx \in \Omega, \quad s = 1,\ldots,m,
\label{eqn:schrodderiv}
\end{equation}
with boundary condition
\begin{equation}
\left[ \bn(\bx) \cdot \nabla - \imath k_0 \right] w^{(s)}(\bx) = \imath u^{(s)}(\bx; k_0), 
\quad \bx \in \partial \Omega, \quad s = 1,\ldots,m.
\label{eqn:bcderiv}
\end{equation}
\end{theorem}
The proof based on the implicit function theorem is given in Appendix~\ref{app:diff}.

\subsection{Measurements and the inverse scattering problem}
\label{sec:measisp}

In practice, performing measurements of the solutions of \eqref{eqn:weaksom} at all
wavenumbers $k \in \R_+$ is infeasible. Thus, the first assumption we make about the measurement
setup is that we only have access to information about the wavefields at a finite number $n$ of 
distinct sampling wavenumbers
\begin{equation}
0 < k_1 < k_2 < \ldots < k_n.
\end{equation}
We use these wavenumbers to define the \emph{snapshots} of wavefields and their derivatives
\begin{equation}
u_j^{(s)}(\bx) = u^{(s)}(\bx; k_j), \quad
\partial_k u_j^{(s)}(\bx) = \partial_k u^{(s)}(\bx; k_j), 
\quad j = 1,\ldots,n, \quad s = 1,\ldots,m,
\label{eqn:usnap}
\end{equation}
Note that in inverse scattering one does not have access to the snapshots, but only to their 
``far field'' measurements, which in our setting correspond to boundary traces that we denote by
\begin{equation}
\phi_j^{(s)} = \left. u_j^{(s)} \right|_{\partial \Omega}, \quad
\partial_k  \phi_j^{(s)} = \left. \partial_k u_j^{(s)} \right|_{\partial \Omega}, \quad
j = 1,\ldots,n, \quad s = 1,\ldots,m.
\label{eqn:thetasnap}
\end{equation}
Thus, we assume that the data has the form 
\begin{equation}
\cD = \left\{ \phi_j^{(s)} , \partial_k \phi_j^{(s)}  \right\}_{j=1,\ldots,n; s=1,\ldots,m}.
\label{eqn:data}
\end{equation}
Practically, the trace derivative data $\partial_k \phi_j^{(s)}$ can be obtained by measuring boundary 
traces $\left. u^{(s)}(\bx; k) \right|_{\partial \Omega}$ at one or more wavenumbers close to $k_j$
followed by interpolation and numerical differentiation. We can now formulate the inverse scattering 
problem that we address here.

\textbf{Inverse scattering problem (ISP)}. Given the data \eqref{eqn:data} estimate the scattering
potential $q\in L^\infty_+(\Omega)$.

\begin{remark}
We use the term ``estimate'' in the formulation of the ISP above since in general it may be impossible
to recover exactly the scattering potential $q$ from the finite amount of data in \eqref{eqn:data}.
\end{remark}

\begin{remark}
As we shall see in what follows, the pointwise knowledge of the traces \eqref{eqn:data} on 
$\partial \Omega$ is not needed for the construction of the ROMs of the kind that we intend to use. 
Instead, the knowledge of certain integrals involving the traces in $\cD$ is sufficient. 
However, to keep the formulation of the ISP concise and clear, we assume that the whole of 
$\cD$ can be measured.
\end{remark}

\subsection{Projection reduced order model}
\label{sec:projrom}

At the heart of the approach proposed here for the numerical solution of the ISP formulated above
are the techniques of projection-based model order reduction. The first step in all such approaches 
\cite{mamonov2022velocity,borcea2014model,druskin2021lippmann} is to project the operators 
\eqref{eqn:op3} onto the reduced order space
\begin{equation}
\mathcal{X}
 = \text{span} \left\{u_j^{(s)} \right\}_{ j = 1,...,n; \ s = 1,...,m }
\label{eqn:romspace}
\end{equation}
spanned by the wavefield snapshots \eqref{eqn:usnap}. If the snapshots for all sampling wavenumbers
and all sources are linearly independent, then
\begin{equation}
\dim(\mathcal{X}) = mn.
\label{eqn:dimx}
\end{equation}

While eventually we are interested in orthogonal projections of \eqref{eqn:op3} onto $\mathcal{X}$, 
we postpone this discussion until Section~\ref{sec:lanczos}. For now, we consider the stiffness 
and mass matrices along with another boundary matrix that \hl{arise} if we view the problem 
\eqref{eqn:opform} in Galerkin framework. Specifically, we introduce the matrices
\begin{equation}
\bS, \bM, \bB \in \C^{mn \times mn},
\label{eqn:mat3}
\end{equation}
referred to as the stiffness, mass and boundary matrices. Hereafter we denote $mn \times mn$
matrices by bold uppercase letters. Due to the indexing of snapshots according to sampling
wavenumbers and source numbers, all three matrices \eqref{eqn:mat3} have a block structure
consisting of $n \times n$ blocks of size $m \times m$ each, e.g.,
\begin{equation}
\bS = \begin{bmatrix}
\bs_{11} & \bs_{12} & \ldots & \bs_{1n} \\ 
\bs_{21} & \bs_{22} & \ldots & \bs_{2n} \\
\vdots & \vdots & \ddots & \vdots \\
\bs_{n1} & \bs_{n2} & \ldots & \bs_{nn} \\
\end{bmatrix} \in \C^{mn \times mn},
\label{eqn:mats}
\end{equation}
with 
\begin{equation}
\bs_{ij} \in \C^{m \times m}, \quad i,j = 1,\ldots,n.
\end{equation}
Likewise,
\begin{align}
\bM & = \left[ \bm_{ij} \right]_{i,j=1,\ldots,n}, \quad \bm_{ij} \in \C^{m \times m}, 
\quad i,j = 1,\ldots,n, \label{eqn:matm} \\
\bB & = \left[ \bb_{ij} \right]_{i,j=1,\ldots,n}, \quad \bb_{ij} \in \C^{m \times m}, 
\quad i,j = 1,\ldots,n. \label{eqn:matb}
\end{align}
We refer to the sub-matrices $\bs_{ij}$, $\bm_{ij}$ and $\bb_{ij}$ in \eqref{eqn:mats}--\eqref{eqn:matb}
as the blocks and hereafter denote them by bold lowercase letters.

Following the standard Galerkin framework, we use the formulas \eqref{eqn:ops}--\eqref{eqn:opb} 
that define the operators \eqref{eqn:op3} to express the blocks of the three matrices \eqref{eqn:mat3} 
in terms of their individual entries
\begin{align}
[\bs_{ij}]_{rs} & = \left\langle \cS u_i^{(r)}, u_j^{(s)} \right\rangle_{H^1(\Omega)^\prime}  = 
 \int_\Omega \overline{\nabla u^{(r)}_i} \cdot \nabla u^{(s)}_j d\bx 
+ \int_\Omega q \overline{ u^{(r)}_i} u^{(s)}_j d\bx, 
\label{eqn:sij} \\
[\bm_{ij}]_{rs} & = \left\langle \cM u_i^{(r)}, u_j^{(s)} \right\rangle_{H^1(\Omega)^\prime}  = 
\int_\Omega \overline{u^{(r)}_i} u^{(s)}_j d\bx, 
\label{eqn:mij} \\
[\bb_{ij}]_{rs} & = \left\langle \cB u_i^{(r)}, u_j^{(s)} \right\rangle_{H^1(\Omega)^\prime}  = 
\int_{\partial \Omega} \overline{u^{(r)}_i} u^{(s)}_j d\Sigma,
\label{eqn:bij}
\end{align}
where the block indices \hl{run over wavenumber indices} $i,j = 1,\ldots,n$, 
while the entries within each block are indexed by the ``source/receiver'' indices $r,s = 1,\ldots,m$. 
Clearly, the formulas \eqref{eqn:sij}--\eqref{eqn:bij} imply that the blocks of the stiffness, mass and
boundary matrices for all $i,j = 1,\ldots,n$, and $r,s = 1,\ldots,m$, satisfy
\begin{align}
[\bs_{ij}]_{rs} & = \overline{[\bs_{ji}]}_{sr},  \\
[\bm_{ij}]_{rs} & = \overline{[\bm_{ji}]}_{sr},  \\
[\bb_{ij}]_{rs} & = \overline{[\bb_{ji}]}_{sr}, 
\label{eqn:hermitianMatrices}
\end{align}
or simply
\begin{equation}
\bs_{ij} = \bs^*_{ji}, \quad \bm_{ij} = \bm^*_{ji}, \quad \bb_{ij} = \bb^*_{ji}, \quad
i,j = 1,\ldots,n,
\label{eqn:mat3hermitblock}
\end{equation}
which means that the three corresponding matrices are Hermitian
\begin{equation}
\bS = \bS^*, \quad \bM = \bM^*, \quad \bB = \bB^*.
\label{eqn:mat3hermit}
\end{equation}
In addition, \eqref{eqn:mat3hermit} in conjunction with \eqref{eqn:mij} and \eqref{eqn:dimx} also 
implies that $\bM$ is positive-definite.

Note that in order to compute the blocks of $\bS$ and $\bM$ using relations 
\eqref{eqn:sij}--\eqref{eqn:mij} one requires the knowledge of the wavefield snapshots in the whole 
$\Omega$, while we are interested in solving the ISP where we only have access to the traces 
\eqref{eqn:data}. Incidentally, the mass and stiffness matrices possess a remarkable property that 
allows their blocks to be computed from the data $\cD$ only. Since the matrices $\bS$, $\bM$ and $\bB$ 
can be thought of as the ROM of the operators $\cS$, $\cM$ and $\cB$, respectively, we will therefore 
refer to such ROM as data-driven. Computing the ROM in the data-driven manner is discussed in detail 
in the next section.

\section{Main results and the method}
\label{sec:mainres}

In this section we introduce the method for the numerical solution of the ISP based on data-driven ROM. 
First, we study the computation of the ROM matrices \eqref{eqn:mat3} from the data \eqref{eqn:data}.
Next, we consider orthogonal projection ROM obtained \hl{by means of the block Lanczos} algorithm. Once the 
ROM is computed, we formulate the optimization problem for solving the ISP via regularized ROM misfit 
minimization.

\subsection{Data-driven ROM}

As mentioned in Section~\ref{sec:projrom} even though the ROM matrices \eqref{eqn:mat3} are 
defined in terms of the wavefield snapshots in the bulk of $\Omega$, it is possible to compute them 
in a data-driven way making them useful for solving the ISP numerically. Assuming we have access 
to snapshot traces in $\cD$, we can compute the following integral quantities. First, we need the blocks
\begin{equation}
\bd_j, \partial_k \bd_j \in \C^{m \times m}, \quad j=1,\ldots,n,
\end{equation}
defined entrywise as
\begin{align}
[\bd_j]_{rs} & 
= \int_{\partial \Omega} \overline{p_r} u_j^{(s)} d \Sigma
= \int_{\partial \Omega} p_r u_j^{(s)} d \Sigma, 
\label{eqn:dblocks} \\
[\partial_k \bd_j]_{rs} & 
= \int_{\partial \Omega} \overline{p_r} \partial_k u_j^{(s)} d \Sigma
= \int_{\partial \Omega} p_r \partial_k u_j^{(s)} d \Sigma,
\label{eqn:dkdblocks}
\end{align}
with $r,s = 1,\ldots,m$ and $j=1,\ldots,n$, 
where we used assumption \eqref{eqn:src} that the sources are real valued. 
This allows us to obtain the following reciprocity result. 
\begin{proposition}
The blocks \eqref{eqn:dblocks} are complex-symmetric
\begin{equation}
\bd_j^T = \bd_j, \quad j = 1,\ldots,n.
\label{eqn:dblocksym}
\end{equation}
\label{prop: complex symmetric data}
\end{proposition}
\begin{proof}
For a fixed $j \in \{ 1,2,\ldots,n \}$ we begin by taking the complex conjugate of \eqref{eqn:weaksom}:
\begin{equation}
\int_\Omega \nabla u^{(s)}_j \cdot \overline{\nabla {v}} \; d \bx + 
\int_{\Omega} q u^{(s)}_j \overline{v} d \bx - 
k_j^2 \int_\Omega u^{(s)}_j \overline{v} d \bx - 
\imath k_j \int_{\partial \Omega} u^{(s)}_j \overline{v} d\Sigma = 
\int_{\partial \Omega} {p_s} \overline{v} d\Sigma, 
\label{eqn:weaksoms}
\end{equation}
then, choosing $v = \overline{u^{(r)}_j}$, we obtain
\begin{equation}
\int_\Omega \nabla u^{(s)}_j \cdot \nabla u^{(r)}_j \; d \bx + 
\int_{\Omega} q u^{(s)}_j u^{(r)}_j d \bx - 
k_j^2 \int_\Omega u^{(s)}_j u^{(r)}_j d \bx - 
\imath k_j \int_{\partial \Omega} u^{(s)}_j u^{(r)}_j d\Sigma = 
\int_{\partial \Omega} {p_s} u^{(r)}_j d\Sigma.
\label{eqn:weaksomsr}
\end{equation}
On the other hand, replacing $s$ with $r$ in \eqref{eqn:weaksoms} and taking the test function
to be $v = \overline{u^{(s)}_j}$, we arrive at
\begin{equation}
\int_\Omega \nabla u^{(r)}_j \cdot \nabla u^{(s)}_j \; d \bx + 
\int_{\Omega} q u^{(r)}_j u^{(s)}_j d \bx - 
k_j^2 \int_\Omega u^{(r)}_j u^{(s)}_j d \bx - 
\imath k_j \int_{\partial \Omega} u^{(r)}_j u^{(s)}_j d\Sigma = 
\int_{\partial \Omega} {p_r} u^{(s)}_j d\Sigma.
\label{eqn:weaksomrs}
\end{equation}
Since the left-hand sides of \eqref{eqn:weaksomsr} and \eqref{eqn:weaksomrs} are equal, 
the same holds for the \hl{right-hand sides}
\begin{equation}
[\bd_j]_{sr} = 
\int_{\partial \Omega} {p_s} u^{(r)}_j d\Sigma = 
\int_{\partial \Omega} {p_r} u^{(s)}_j d\Sigma =
[\bd_j]_{rs}, \quad s,r = 1,\ldots,m,
\label{eqn:dj}
\end{equation}
which immediately implies \eqref{eqn:dblocksym}.
\end{proof}
Note that since 
\begin{equation}
[\partial_k \bd_j]_{rs} 
= \int_{\partial \Omega} p_r \partial_k u_j^{(s)} d \Sigma
= \left. \partial_k \left[ \int_{\partial \Omega} p_r(\bx) u^{(s)}(\bx; k) d \Sigma(\bx) \right] \right|_{k = k_j},
\end{equation}
it follows from \eqref{eqn:dblocks} and \eqref{eqn:dblocksym} that the derivative blocks \eqref{eqn:dkdblocks}
are complex-symmetric as well
\begin{equation}
[\partial_k \bd_j]^T = \partial_k \bd_j, \quad j = 1,\ldots,n.
\label{eqn:dkdblocksym}
\end{equation}

We also need the boundary matrix and blocks
\begin{equation}
\bB \in \C^{mn \times mn}, \quad \bc_j \in \C^{m \times m}, \quad j=1,\ldots,n,
\end{equation}
defined as
\begin{align}
[\bb_{ij}]_{rs} & 
= \int_{\partial \Omega} \overline{u^{(r)}_i} u^{(s)}_j d\Sigma, 
\label{eqn:bblocks} \\
[\bc_{j}]_{rs} &
=  \int_{\partial \Omega} \left[- \overline{u^{(r)}_j } \partial_k u^{(s)}_j
+ u^{(s)}_j \overline{\partial_k u^{(r)}_j} \right] d\Sigma,
\label{eqn:cblocks}
\end{align}
with $i,j = 1,\ldots,n$, and $r,s = 1,\ldots,m$. As established already in Section~\ref{sec:projrom},
the blocks \eqref{eqn:bblocks} are Hermitian and so is $\bB$. It follows directly from the definition
\eqref{eqn:cblocks} that the boundary blocks $\bc_j$ are skew-Hermitian
\begin{equation}
\bc_{j}^* = -\bc_{j}, \quad j=1,\ldots,n.
\label{eqn:cblockasym}
\end{equation}

As explained later, the stiffness and mass matrices, $\bS$ and $\bM$, respectively, are completely 
determined by the boundary integral quantities \eqref{eqn:dblocks}--\eqref{eqn:dkdblocks} and 
\eqref{eqn:bblocks}--\eqref{eqn:cblocks}. Thus, the knowledge of the whole data set $\cD$ 
is not needed for our approach to solving the ISP. Instead, the knowledge of 
\begin{equation}
\{ \bB, \bc_j, \bd_j, \partial_k \bd_j \}_{j=1,\ldots,n}
\label{eqn:bdryint}
\end{equation}
is sufficient. However, for the sake of simplicity we stick to the ISP formulation from 
Section~\ref{sec:measisp} since the data $\cD$ can be easily converted to the knowledge of
\eqref{eqn:bdryint} using \eqref{eqn:dblocks}--\eqref{eqn:dkdblocks} and 
\eqref{eqn:bblocks}--\eqref{eqn:cblocks}.

\subsubsection{Data-driven stiffness and mass matrix computaion}

We provide here the formulas for computing the blocks \eqref{eqn:sij}--\eqref{eqn:mij} 
of the stiffness and mass matrices from \eqref{eqn:bdryint} and thus from the data \eqref{eqn:data}.
Note that the formulas for the off-diagonal and diagonal blocks are different, hence we formulate
them as separate propositions below. The proofs of all the propositions that follow can be found in 
Appendix~\ref{app:proofs}.

\begin{proposition}
\label{prop:sij}
The off-diagonal blocks $\bs_{ij} \in \C^{m \times m}$, $i,j = 1,\ldots,n$, $i \neq j$, 
of the stiffness matrix \eqref{eqn:sij} are given by
\begin{equation}
\bs_{ij} = 
\frac{k^2_i \bd_{i}^* - k^2_j {\bd_{j}}}{k^2_i - k^2_j}
- \imath \frac{  (k_i k^2_j + k^2_i k_j ) \bb_{ij}}{k^2_i - k^2_j}.
\label{eqn:ddsij}
\end{equation}
\end{proposition}

\begin{proposition}
\label{prop:sjj}
The diagonal blocks $\bs_{jj} \in \C^{m \times m}$, $j = 1,\ldots,n$, 
of the stiffness matrix \eqref{eqn:sij} are given by
\begin{equation}
\bs_{jj} = \frac{1}{2} \Big(  k_j \Re ( \partial_k \bd_j )+ 2 \Re(\bd_j) \Big) 
+ \frac{\imath k_j^2}{2} \bc_j.
\label{eqn:ddsjj}
\end{equation}
\end{proposition}

\begin{proposition}
\label{prop:mij}
The off-diagonal blocks $\bm_{ij} \in \C^{m \times m}$, $i,j = 1,\ldots,n$, $i \neq j$, 
of the mass matrix \eqref{eqn:mij} are given by
\begin{equation}
\bm_{ij} = \frac{\bd_{i}^* - \bd_{j}}{k^2_i - k^2_j} 
- \imath \frac{\bb_{ij} }{k_j - k_i}.
\label{eqn:ddmij}
\end{equation}
\end{proposition}

\begin{proposition}
\label{prop:mjj}
The diagonal blocks $\bm_{jj} \in \C^{m \times m}$, $j = 1,\ldots,n$, 
of the mass matrix \eqref{eqn:mij} are given by
\begin{equation}
\bm_{jj} = \frac{1}{2 k_j} \Re ( \partial_k \bd_j )
+ \frac{\imath}{2} \bc_j.
\label{eqn:ddmjj}
\end{equation}
\end{proposition}

We summarize data-driven stiffness and mass matrix computation in Algorithm~\ref{alg:ddrom} below.

\begin{algorithm}
\caption{Data-driven stiffness and mass matrix computation}
\label{alg:ddrom}
\begin{algorithmic}
\State \textbf{Input:} boundary data 
$\cD = \left\{ \phi_j^{(s)} , \partial_k \phi_j^{(s)}  \right\}_{j=1,\ldots,n; s=1,\ldots,m}$.
\State $\bullet$ Compute the blocks
\begin{align}
[\bb_{ij}]_{rs} & 
= \int_{\partial \Omega} \overline{\phi^{(r)}_i} \phi^{(s)}_j d\Sigma, 
\label{eqn:abblocks} \\
[\bc_{j}]_{rs} &
=  \int_{\partial \Omega} \left[ -\overline{\phi^{(r)}_j } \partial_k \phi^{(s)}_j
+ \phi^{(s)}_j \overline{\partial_k \phi^{(r)}_j} \right] d\Sigma,
\label{eqn:acblocks} \\
[\bd_j]_{rs} & 
= \int_{\partial \Omega} p_r \phi_j^{(s)} d \Sigma, 
\label{eqn:adblocks} \\
[\partial_k \bd_j]_{rs} & 
= \int_{\partial \Omega} p_r \partial_k \phi_j^{(s)} d \Sigma,
\label{eqn:adkdblocks}
\end{align}
for $i,j = 1,\ldots,n$, $r,s = 1,\ldots,m$.
\State $\bullet$ Compute the blocks $\bs_{ij} \in \C^{m \times m}$ and $\bm_{ij}$, $i,j = 1,\ldots,n$, 
using formulas \eqref{eqn:ddsij}--\eqref{eqn:ddmjj}.
\State \textbf{Output:} mass and stiffness matrices $\bM$ and $\bS$.
\end{algorithmic}
\end{algorithm}

\subsubsection{Orthogonal projection ROM via block-Lanczos algorithm}
\label{sec:lanczos}

In view of relation  \eqref{eqn:sij}, we observe that the dependence of the stiffness matrix $\bS$ on the 
potential of interest $q$ is affine provided the snapshots are known. This dependence is 
approximately affine if the snapshots depend weakly on $q$. If neither is the case, one may 
wish to transform the ROM matrix $\bS$ further so that it corresponds to an orthogonal projection 
of $\cS$ onto the reduced order space \eqref{eqn:romspace}. This requires introduction of the so-called 
orthogonalized snapshots $v_j^{(s)}$ that form an orthonormal basis for the reduced space $\mathcal{X}$:
\begin{equation}
\mathcal{X} = \text{span} \left\{v_j^{(s)} \right\}_{ j = 1,...,n; \ s = 1,...,m },
\label{eqn:vspan}
\end{equation}
where orthonormality is understood as
\begin{equation}
\int_\Omega \overline{v_i^{(r)}} v_j^{(s)} d \bx = \delta_{ij} \delta_{rs}, \quad
i,j = 1,\ldots,n, \quad r,s=1,\ldots,m.
\label{eqn:vorth}
\end{equation}

We are interested in a specific orthonormal basis so that the projection of $\cS$ onto $\mathcal{X}$ in this basis
has the block-tridiagonal form
\begin{equation}
\bT = \begin{bmatrix}
\balpha_1  & \bbeta_2 & \bzero & \hdots & \bzero \\
\bbeta_2^* & \balpha_2 & \bbeta_3 & \ddots & \vdots \\
\bzero & \bbeta_3^* & \ddots & \ddots & \bzero \\
\vdots & \ddots & \ddots & \balpha_{n-1} & \bbeta_n \\
\bzero & \hdots & \bzero & \bbeta_n^* & \balpha_n
\end{bmatrix} \in \C^{mn \times mn},
\end{equation}
with blocks 
$\balpha_i, \bbeta_j \in \C^{m \times m}$ given entrywise by
\begin{align}
[\balpha_{i}]_{rs} & = \left\langle \cS v_i^{(r)}, v_i^{(s)} \right\rangle_{H^1(\Omega)^\prime}  = 
 \int_\Omega \overline{\nabla v^{(r)}_i} \cdot \nabla v^{(s)}_i d\bx 
+ \int_\Omega q \overline{ v^{(r)}_i} v^{(s)}_i d\bx, & i=1,\ldots,n, \\
[\bbeta_{j}]_{rs} & = \left\langle \cS v_{j-1}^{(r)}, v_{j}^{(s)} \right\rangle_{H^1(\Omega)^\prime}  = 
 \int_\Omega \overline{\nabla v^{(r)}_{j-1}} \cdot \nabla v^{(s)}_{j} d\bx 
+ \int_\Omega q \overline{ v^{(r)}_{j-1}} v^{(s)}_{j} d\bx, & j=2,\ldots,n.
\end{align}
Such a basis can be computed by applying the block-Lanczos process to the matrix
\begin{equation}
\widetilde{\bS} = \bM^{-1/2} \bS \bM^{-1/2} \in \C^{mn \times mn}
\label{eqn:stilde}
\end{equation}
and the starting column-block
\begin{equation}
\tbd = \hl{\bM^{1/2}}
\begin{bmatrix}
\overline{\bd_1} \\ \overline{\bd_2} \\ \vdots \\ \overline{\bd_n}
\end{bmatrix} \in \C^{mn \times m}.
\label{eqn:btilde}
\end{equation}

Note that since $\bM$ is positive-definite, as mentioned in Section~\ref{sec:projrom}, there exists 
a unique positive-definite matrix square root $\bM^{1/2} \in \C^{mn \times mn}$. 
The presence of the terms $\bM^{-1/2}$ and $\bM^{1/2}$ in \eqref{eqn:stilde} and \eqref{eqn:btilde},
respectively, enables the use of the regular block-Lanczos process based on the standard 
$\C^{mn}$-inner product. 
Alternatively, one may apply block-Lanczos directly to ${\bS}$ and the starting column-block 
$[\bd_1, \bd_2, \ldots, \bd_n]^*$, but then a modification is needed for all inner products in 
block-Lanczos process to be computed in $\bM$-weighted inner product. In any case, the resulting 
matrix $\bT$ is the same, so we opt for a simpler variant of block-Lanczos as discussed below.

The block-Lanczos process given in Algorithm~\ref{alg:lanczos} computes a block-tridagonal matrix 
$\bT$ and a unitary matrix $\bQ \in \C^{mn \times mn}$ that can be written in block-column form
\begin{equation}
\bQ = [\bq_1, \bq_2, \ldots, \bq_{n}], \quad
\bq_j \in \C^{mn \times m}, \quad 
j = 1,\ldots,n,
\end{equation}
satisfying
\begin{equation}
\bq_i^* \bq_j = \delta_{ij} \bI_m \in \C^{m \times m}, \quad
i, j = 1,\ldots,n,
\label{eqn:qorth}
\end{equation}
where $\bI_m$ is the $m \times m$ identity matrix. The matrix $\bQ$ defines a unitary change of 
coordinates such that
\begin{equation}
\bT = \bQ^* \widetilde{\bS} \bQ = \bQ^* \bM^{-1/2} \bS \bM^{-1/2} \bQ.
\label{eqn:tlanczos}
\end{equation}
\hl{Algorithm~\ref{alg:lanczos} contains an optional reorthogonalization step for increased numerical 
stability. In practice, the computational cost of this step is negligible compared to the overall cost of 
solving the ISP using the transformed ROM matrix $\bT$. Therefore, we perform reorthogonalization at 
every step of Algorithm~\ref{alg:lanczos}.}

Relation \eqref{eqn:tlanczos} implies an explicit formula for the orthogonalized snapshots.
Introducing notation
\begin{equation}
\bu_j(\bx) = \left[ u^{(1)}_j(\bx), u^{(2)}_j(\bx), \ldots, u^{(m)}_j(\bx) \right], \quad
\bv_j(\bx) = \left[ v^{(1)}_j(\bx), v^{(2)}_j(\bx), \ldots, v^{(m)}_j(\bx) \right], 
\end{equation}
with $\bu_j, \bv_j: \Omega \to \C^{1 \times m}$, $j = 1,\dots,n$, we can express the orthogonalized 
snapshots as
\begin{equation}
\left[ \bv_1(\bx) ,\ldots, \bv_n(\bx) \right] = 
\left[ \bu_1(\bx) ,\ldots, \bu_n(\bx) \right] \bM^{-1/2} \bQ, \quad
\bx \in \Omega.
\end{equation}
Observe that \eqref{eqn:mij} can be written as
\begin{equation}
\int_\Omega \bu_i^*(\bx) \bu_j(\bx) d\bx = \bm_{ij}, \quad i,j=1,\ldots,n,
\end{equation}
which in conjunction with \eqref{eqn:qorth} implies
\begin{equation}
\int_\Omega \bv_i^*(\bx) \bv_j(\bx) d\bx = \delta_{ij} \bI_m, \quad i,j=1,\ldots,n,
\end{equation}
which is a matrix form of orthogonality relations \eqref{eqn:vorth}. 

\begin{algorithm}
\caption{Block Lanczos Process}
\label{alg:lanczos}
\begin{algorithmic}
\State \textbf{Input:} matrix $\widetilde{\bS} \in \C^{mn \times mn}$ 
and starting block-column ${\tbd} \in \C^{mn \times m}$. \\
$\bullet$ \text{Set} $\bbeta_1 = \left( \tbd^ * \tbd \right)^{1/2}$;\\
$\bullet$ \text{Set} $\bq_1 = \tbd \ \bbeta_1^{-1}$;\\
$\bullet$ \text{Set} $\bw = \widetilde{\bS} \bq_1$;
\For{$j=1,...,n-1 $}
\begin{enumerate}
\item \text{Set} $\balpha_j = \bw^ * \bq_j$;
\item \text{Update} $\bw = \bw - \bq_j \balpha_j$;
\item \text{If needed, perform reorthogonalization by updating}
\begin{equation*}
\bw = \bw - [\bq_1,\ldots,\bq_j] \left( [\bq_1,\ldots,\bq_j]^* \bw \right);
\end{equation*}
\item \text{Set} $\bbeta_{j+1} = (\bw^* \bw)^{1/2}$;
\item \text{Set} $\bq_{j+1} = \bw \ \bbeta_{j+1}^{-1}$;
\item \text{Set} $\bw = \widetilde{\bS} \bq_{j+1} - \bq_j \bbeta_{j+1}$;
\end{enumerate}
\EndFor
\\
$\bullet$ \text{Set} $\balpha_{n} = \bw^* \bq_n$;\\
\textbf{Output:} the blocks $\balpha_j, \bbeta_j \in \C^{m \times m}$, $j=1,\ldots,n$, of $\bT$ 
and the unitary matrix $\bQ = [\bq_1, \ldots, \bq_n] \in \C^{mn \times mn}$.
\end{algorithmic}
\end{algorithm}

We conclude this section with Algorithm~\ref{alg:ddt} for computing the block-tridiagonal matrix $\bT$
in a data-driven manner.

\begin{algorithm}
\caption{Data-driven block tridiagonal matrix computation}
\label{alg:ddt}
\begin{algorithmic}
\State \textbf{Input:} boundary data 
$\cD = \left\{ \phi_j^{(s)} , \partial_k \phi_j^{(s)}  \right\}_{j=1,\ldots,n; s=1,\ldots,m}$.
\State $\bullet$  Use Algorithm~\ref{alg:ddrom} to compute 
$\bS \in \C^{mn \times mn}$ and $\bM \in \C^{mn \times mn}$ from $\cD$.
\State $\bullet$ Form $\widetilde{\bS} \in \C^{mn \times mn}$ and 
$\widetilde{\bd} \in \C^{mn \times m}$ using \eqref{eqn:stilde} and \eqref{eqn:btilde}, respectively.
\State $\bullet$ Apply Algorithm~\ref{alg:lanczos} to $\widetilde{\bS}$ and $\widetilde{\bd}$
to compute $\bT \in \C^{mn \times mn}$.
\State \textbf{Output:} block-tridiagonal $\bT \in \C^{mn \times mn}$.
\end{algorithmic}
\end{algorithm}

\subsection{Numerical solution of ISP via ROM misfit minimization}

Conventionally, one may solve the ISP numerically in the following manner. 
Consider the mapping from the potential $q$ to the boundary data $\cD$ defined by expressions
\eqref{eqn:schrodom}--\eqref{eqn:bc}, \eqref{eqn:schrodderiv}--\eqref{eqn:bcderiv},
\eqref{eqn:usnap}--\eqref{eqn:data} and denoted by
\begin{equation}
\cD[q] = \left\{ \phi_j^{(s)}[q], \partial_k \phi_j^{(s)}[q] \right\}_{j=1,\ldots,n; s=1,\ldots,m}.
\label{eqn:dataq}
\end{equation}
Then, a typical approach to numerical solution of the ISP is via non-linear least squares
\begin{equation}
\mathop{\text{minimize}}\limits_{\hq \in \cQ} \sum_{j=1}^{n} \sum_{s=1}^{m} 
\int_{\partial \Omega} \left|  \phi_j^{(s)} - \phi_j^{(s)}[\hq] \right|^2 d \Sigma,
\label{eqn:lsdata}
\end{equation}
where $\phi_j^{(s)}$ is the measured data, $\hq$ is the search potential and $\cQ$ is 
some a priori chosen search space, typically finite-dimensional.

The formulation \eqref{eqn:lsdata} is known to have its limitations including possibly
slow convergence and lack of robustness with respect to the initial guess due to non-convexity
of the objective, see for example \cite{10.1093/gji/ggy380,doi:10.1190/geo2020-0851.1} and the 
references therein. On the other hand, extensive studies of using data-driven ROMs for estimating 
the coefficients of wave equations in time domain
suggest that reformulating \eqref{eqn:lsdata} as a ROM misfit minimization often leads to objective
convexification consequently improving convergence and robustness. Thus, we propose two 
possible ROM-based approaches to the numerical solution of the ISP. Both approaches have the form
\begin{equation}
\mathop{\text{minimize}}\limits_{\hq \in \cQ} \cF(\hq) + \mu \cR(\hq),
\label{eqn:lsrom}
\end{equation}
where $\cF(\hq)$ is the ROM misfit functional, $\cR(\hq)$ is the regularization functional and
$\mu > 0$ is the regularization parameter. The two possible choices of ROM misfit functionals are
\begin{equation}
\cF_\bS(\hq) = \big\| \text{Triu}( \bS^{\text{\scriptsize meas}} - \bS[\hq] ) \big\|_2^2
\label{eqn:funs}
\end{equation}
and
\begin{equation}
\cF_\bT(\hq) = \big\| \text{Triu}( \bT^{\text{\scriptsize meas}} - \bT[\hq] ) \big\|_2^2.
\label{eqn:funt}
\end{equation}
The notation in \eqref{eqn:funt}--\eqref{eqn:funs} is as follows. We denote by 
$\text{Triu}: \C^{mn \times mn} \to \C^{mn (mn + 1) /2}$ the operation of taking the entries 
in the upper triangular part of a matrix and stacking them in a vector. 
The stiffness and block-tridiagonal
matrices $\bS^{\text{\scriptsize meas}}$ and $\bT^{\text{\scriptsize meas}}$, respectively, 
are computed from the measured, possibly noisy data $\cD^{\text{\scriptsize meas}}$, 
while their analogues $\bT[\hq]$ and $\bS[\hq]$ are computed from $\cD[\hq]$ corresponding 
to the search potential $\hq$.

Before continuing further, we should discuss the effects of noise in the data on solving the ISP using ROMs. 
We define the noisy measurement data $\cD^{\text{\scriptsize meas}}$ similary to noiseless 
data \eqref{eqn:data}:
\begin{equation}
\cD^{\text{\scriptsize meas}} = \left\{ 
\phi_j^{\text{\scriptsize meas},(s)} = \phi_j^{(s)} + \xi_j^{(s)}, 
\partial_k \phi_j^{\text{\scriptsize meas},(s)} = \partial_k \phi_j^{(s)} + \chi_j^{(s)} 
\right\}_{j=1,\ldots,n; s=1,\ldots,m},
\label{eqn:datan}
\end{equation}
with the exception of adding the terms $\xi_j^{(s)}$ and $\chi_j^{(s)}$ that model the noise in the 
measurements of the wavefield and its derivative with respect to $k$ at the boundary $\partial \Omega$.
Note that the effect of noise on constructing $\bS^{\text{\scriptsize meas}}$ and 
$\bT^{\text{\scriptsize meas}}$ from \eqref{eqn:datan} could be profound. 
In particular, as discussed in Section~\ref{sec:projrom} the mass matrix is 
Hermitian positive definite, and so is the stiffness matrix provided $q(\bx) \geq 0$. This property may be 
broken if $\bM^{\text{\scriptsize meas}}$ and $\bS^{\text{\scriptsize meas}}$ are assembled from noisy 
$\cD^{\text{\scriptsize meas}}$ using Algorithm~\ref{alg:ddrom}. Thus, a modification of 
\eqref{eqn:funs}--\eqref{eqn:funt} is needed to accommodate the noise in \eqref{eqn:datan}.

First, consider modifying \eqref{eqn:funs} to work with noisy data. One may notice that the spectrum
of both $\bM^{\text{\scriptsize meas}}$ and $\bS^{\text{\scriptsize meas}}$ is not affected by the 
noise uniformly. Instead, \hl{for realistic noise models (e.g., the noise model described in Section~\ref{sec:num})},
the small eigenvalues of $\bM^{\text{\scriptsize meas}}$ and $\bS^{\text{\scriptsize meas}}$ are more 
sensitive and thus may cross over into the negative values making the mass and stiffness matrices indefinite. Therefore, one possible approach to modifying
\eqref{eqn:funs} is via spectral projection. Consider the spectral decomposition of the stiffness matrix
\begin{equation}
\bS^{\text{\scriptsize meas}} = \bZ_\bS \bLambda_\bS \bZ_\bS^*,
\label{eqn:eigens}
\end{equation}
where $\bLambda_\bS = \text{diag} \left( \lambda^\bS_1, \lambda^\bS_2, \ldots, \lambda^\bS_{mn} \right)$ 
with eigenvalues sorted as $\lambda^\bS_1 \geq \lambda^\bS_2 \geq \ldots \geq \lambda^\bS_{mn}$, 
and the columns of $\bZ_\bS  \in \C^{mn \times mn}$ are the corresponding eigenvectors of 
$\bS^{\text{\scriptsize meas}}$. Note that due to the presence of noise in 
$\cD^{\text{\scriptsize meas}}$, it is possible for some eigenvalues in \eqref{eqn:eigens}
to be negative. Then, modification of \eqref{eqn:funs} is needed when $\lambda^\bS_{mn} < 0$.
In both cases we choose an integer \hl{$r_{\bS}$} such that
\begin{equation}
r_\bS = \left\{
\begin{array}{rl}
\max\limits_{\lambda^\bS_{mr} \geq |\lambda^\bS_{mn}|} r, & \text{ if } \lambda^\bS_{mn} < 0, \\
n, & \text{ otherwise.}
\end{array}
\right.
\label{eqn:rs}
\end{equation}
This heuristic allows to identify the number of eigenvalues and eigenvectors of 
$\bS^{\text{\scriptsize meas}}$ that are less sensitive to noise. 
Denote by $\bZ_\bS^r \in \C^{mn \times mr_\bS}$ the sub-matrix of $\bZ_\bS$ containing its
first $mr_\bS$ columns. Then, the corresponding stable subspace is 
\begin{equation}
\cZ_\bS^r = \text{colspan}(\bZ_\bS^r),
\end{equation}
with the corresponding spectral projector
\begin{equation}
\bP_\bS^r = \bZ_\bS^r \big( \bZ_\bS^r \big)^*.
\label{eqn:projsr}
\end{equation}
Note that $\bP_\bS^n = \bI_{mn}$, i.e., no projection is needed when $r = n$.
Once the stable subspace is computed, the modified objective \eqref{eqn:funs} becomes
\begin{equation}
\cF_\bS^r(\hq) = \big\| \text{Triu} 
\left[ \bP_\bS^r \left( \bS^{\text{\scriptsize meas}} - \bS[\hq] \right) \bP_\bS^r \right] \big\|_2^2.
\label{eqn:funsr}
\end{equation}

Second, modifying \eqref{eqn:funt} to work with noisy data requires more effort. In this case we 
begin with the spectral decomposition of the mass matrix
\begin{equation}
\bM^{\text{\scriptsize meas}} = \bZ_\bM \bLambda_\bM \bZ_\bM^*,
\label{eqn:eigenm}
\end{equation}
with $\bLambda_\bM = \text{diag} \left( \lambda^\bM_1, \lambda^\bM_2, \ldots, \lambda^\bM_{mn} \right)$ 
and $\lambda^\bM_1 \geq \lambda^\bM_2 \geq \ldots \geq \lambda^\bM_{mn}$. 
Similarly to \eqref{eqn:rs}, we take 
\begin{equation}
r_\bM = \left\{
\begin{array}{rl}
\max\limits_{\lambda^\bM_{mr} \geq |\lambda^\bM_{mn}|} r, & \text{ if } \lambda^\bM_{mn} < 0, \\
n, & \text{ otherwise,}
\end{array}
\right.
\label{eqn:lambdarm}
\end{equation}
and define the stable subspace
\begin{equation}
\cZ_\bM^r = \text{colspan}(\bZ_\bM^r),
\end{equation}
where $\bZ_\bM^r \in \C^{mn \times mr_\bM}$ contains the first $mr_\bM$ columns of 
$\bZ_\bM$. Unlike the construction of \eqref{eqn:funsr}, a number of additional steps is required here.
In particular, relations \eqref{eqn:stilde}--\eqref{eqn:btilde} need to be modified as follows.
Denoting 
\begin{equation}
\bLambda^r_\bM = 
\text{diag} \left( \lambda^\bM_1, \lambda^\bM_2, \ldots, \lambda^\bM_{mr_\bM} \right) 
\in \R^{mr_\bM \times mr_\bM},
\label{eqn:rm}
\end{equation}
we introduce
\begin{equation}
\widetilde{\bS}^r = 
(\bLambda^r_\bM)^{-1/2} (\bZ_\bM^r)^* \; \bS^{\text{\scriptsize meas}} \; \bZ_\bM^r (\bLambda^r_\bM)^{-1/2} 
\in \C^{m r_\bM \times m r_\bM},
\label{eqn:stilder}
\end{equation}
and
\begin{equation}
\tbd^r = \hl{(\bLambda^r_\bM)^{1/2}} (\bZ_\bM^r)^*
\begin{bmatrix}
\overline{\bd_1}^{\text{\scriptsize meas}} \\ \overline{\bd_2}^{\text{\scriptsize meas}} \\ \vdots \\ \overline{\bd_n}^{\text{\scriptsize meas}}
\end{bmatrix} \in \C^{m r_\bM \times m},
\label{eqn:btilder}
\end{equation}
the analogues of \eqref{eqn:stilde} and \eqref{eqn:btilde}, respectively. 
Then, to obtain the modified block tridiagonal matrix $\bT^r$, an analogue of Algorithm~\ref{alg:ddt}
can be formulated as Algorithm~\ref{alg:ddtr} below.

The computation of $\bT[\hq]$ for the search potential $\hq$ should also be modified accordingly to the 
construction of $\bT^r$, as outlined below. Denote by $\bM[\hq]$ and $\bS[\hq]$ the stiffness and mass 
matrices corresponding to the search potential $\hq$. Once the stable subspace is found in 
Algorithm~\ref{alg:ddtr}, compute the projection
\begin{equation}
\bM^r[\hq] = (\bZ_\bM^r)^* \; \bM[\hq] \; \bZ_\bM^r,
\label{eqn:mqr}
\end{equation}
and use it to calculate
\begin{equation}
\widetilde{\bS}^r[\hq] = 
(\bM^r[\hq])^{-1/2} (\bZ_\bM^r)^* \; \bS[\hq] \; \bZ_\bM^r (\bM^r[\hq])^{-1/2} 
\in \C^{m r_\bM \times m r_\bM},
\label{eqn:stildeqr}
\end{equation}
and
\begin{equation}
\tbd^r[\hq] = \hl{\left(\bM^r[\hq]\right)^{1/2}} (\bZ_\bM^r)^*
\begin{bmatrix}
\overline{\bd_1}[\hq] \\ \overline{\bd_2}[\hq] \\ \vdots \\ \overline{\bd_n}[\hq]
\end{bmatrix} \in \C^{m r_\bM \times m}.
\label{eqn:dtildeqr}
\end{equation}
Then, Algorithm~\ref{alg:lanczos} can be applied to $\widetilde{\bS}^r[\hq]$ and $\widetilde{\bd}^r[\hq]$
with $n = r_\bM$ to compute $\bT^r[\hq] \in \C^{m r_\bM \times m r_\bM}$. 
The modification of \eqref{eqn:funt} then takes the form
\begin{equation}
\cF_\bT^r(\hq) = \big\| \text{Triu}( \bT^r - \bT^r[\hq] ) \big\|_2^2.
\label{eqn:funtr}
\end{equation}

\begin{algorithm}
\caption{Data-driven block tridiagonal matrix computation from noisy data}
\label{alg:ddtr}
\begin{algorithmic}
\State \textbf{Input:} noisy data $\cD^{\text{\scriptsize meas}}$ as in \eqref{eqn:datan}.
\State $\bullet$  Use Algorithm~\ref{alg:ddrom} to compute 
$\bS^{\text{\scriptsize meas}} \in \C^{mn \times mn}$ and 
$\bM^{\text{\scriptsize meas}} \in \C^{mn \times mn}$ from $\cD^{\text{\scriptsize meas}}$.
\State $\bullet$ Symmetrize the blocks of $\bM^{\text{\scriptsize meas}}$ and $\bS^{\text{\scriptsize meas}}$
to enforce \eqref{eqn:hermitianMatrices} in case those relations were broken by the presence of noise.
\State $\bullet$ Compute the eigendecomposition \eqref{eqn:eigenm}, choose $r_\bM$ as in \eqref{eqn:rm} and 
form the matrices $\bZ_\bM^r$ and $\bLambda_\bM^r$.
\State $\bullet$ Form $\widetilde{\bS}^r \in \C^{m r_\bM \times m r_\bM}$ and 
$\widetilde{\bd}^r \in \C^{m r_\bM \times m}$ using \eqref{eqn:stilder} and \eqref{eqn:btilder}, respectively.
\State $\bullet$ Apply Algorithm~\ref{alg:lanczos} to $\widetilde{\bS}^r$ and $\widetilde{\bd}^r$
with $n = r_\bM$ to compute $\bT^r \in \C^{m r_\bM \times m r_\bM}$.
\State \textbf{Output:} block-tridiagonal $\bT^r \in \C^{m r_\bM \times m r_\bM}$ and the basis
for the stable subspace contained in $\bZ_\bM^r \in \C^{m n \times m r_\bM}$. 
\end{algorithmic}
\end{algorithm}

As mentioned above, we consider a finite-dimensional search space 
\begin{equation}
\cQ = \text{span} \{ \eta_1(\bx), \ldots, \eta_N(\bx) \},
\end{equation}
with $N < mn (mn + 1) /2$, so that the search potential has the form
\begin{equation}
\hq(\bx; \by) = \sum_{l = 1}^{N} y_l \eta_l(\bx),
\label{eqn:qxy}
\end{equation}
parameterized by the coefficients $\by = [y_1, \ldots, y_N]^T \in \R^{N}$. 
Given the parameterization \eqref{eqn:qxy} we employ Tikhonov-like regularization 
\begin{equation}
\cR(\hq(\cdot;\by)) = \| \by \|_2^2
\end{equation}
with an adaptively chosen regularization parameter $\mu$.

In order to solve \eqref{eqn:lsrom} we employ a variant of regularized Gauss-Newton method
summarized in Algorithm~\ref{alg:gn} below. The algorithm can be applied to minimize either
\eqref{eqn:funtr} or \eqref{eqn:funsr} where we let $\bX \in \{ \bT, \bS \}$ to denote the choice 
between the two. In case of noiseless data, one may simply set $r_\bM = r_\bS = n$. 

\begin{algorithm}
\caption{Gauss-Newton method for solving ISP via regularized ROM mifit minimization}
\label{alg:gn}
\begin{algorithmic}
\State \textbf{Input:} sampling wavenumbers $0 < k_1 < k_2 < \ldots < k_n$ and the
corresponding measured boundary data $\cD^{\text{\scriptsize meas}}$,
initial guess for potential coefficients $\by^{(0)} \in \R^N$, maximum iteration number 
$n_{\text{\scriptsize iter}}$, adaptive regularization parameter $\gamma \in (0, 1)$.
\State $\bullet$ If $\bX = \bS$, use Algorithm~\ref{alg:ddrom} to compute the measured matrix $\bS^{\text{\scriptsize meas}}$ from $\cD^{\text{\scriptsize meas}}$; compute the eigendecomposition
\eqref{eqn:eigens} and use it to find $r = r_\bS$ via \eqref{eqn:rs}, the basis for the stable subspace
$\bZ_\bS^r$ and the spectral projector $\bP_\bS^r$ as in \eqref{eqn:projsr};
\State $\bullet$ If $\bX = \bT$, use Algorithm~\ref{alg:ddtr} to compute the matrix 
$\bT^{r, \text{\scriptsize meas}}$ from $\cD^{\text{\scriptsize meas}}$, as well as $r = r_\bM$
and the the basis for the stable subspace $\bZ_\bM^r$;
\For{$i = 1,2,\ldots,n_{\text{\scriptsize iter}}$}
\begin{enumerate}
\item Solve the forward problems \eqref{eqn:schrodom}--\eqref{eqn:bc} and 
\eqref{eqn:schrodderiv}--\eqref{eqn:bcderiv} with $q = \hq \left(\cdot; \by^{(i-1)} \right)$,
for all sampling wavenumbers to generate the corresponding boundary data $\cD^{(i-1)}$;
\item
Compute $\bM^{(i-1)}$ and $\bS^{(i-1)}$ from $\cD^{(i-1)}$ using Algorithm~\ref{alg:ddrom};
\begin{itemize}
\item If $\bX = \bS$, evaluate the residual 
$\br^{(i-1)} = \text{Triu}\big( \bP_\bS^r (\bS^{\text{\scriptsize meas}} - \bS^{(i-1)}) \bP_\bS^r \big)$;
\item If $\bX = \bT$, compute $\widetilde{\bS}^{r, (i-1)}$ and $\widetilde{\bd}^{r, (i-1)}$ as in
\eqref{eqn:stildeqr} and \eqref{eqn:dtildeqr}, respectively, and apply Algorithm~\ref{alg:lanczos}
with $n = r_\bM$ to compute $\bT^{r, (i-1)}$, then evaluate the residual 
$\br^{(i-1)} = \text{Triu}( \bT^{r, \text{\scriptsize meas}} - \bT^{r, (i-1)} )$;
\end{itemize}
\item Compute the Jacobian
\begin{equation}
\bJ^{(i)} = \nabla_\by \cF_{\bX}^r \left( \hq \left( \cdot; \by^{(i-1)} \right) \right) 
\in \C^{mr (mr + 1) /2 \times N},
\end{equation}
where $\bX \in \{ \bT, \bS \}$;
\item Compute the singular values 
$\sigma_1^{(i)} \geq \sigma_2^{(i)} \geq \ldots \geq \sigma_N^{(i)}$
of $\bJ^{(i)}$ and set Tikhonov regularization parameter to
\begin{equation}
\mu^{(i)} = \left( \sigma^{(i)}_{\lfloor \gamma N \rfloor} \right)^2;
\label{eqn:muadap}
\end{equation}
\item Compute the update direction
\begin{equation}
\bz^{(i)} = - \Re\left[ \left(  \left( \bJ^{(i)} \right)^* \bJ^{(i)} + \mu^{(i)} \bI_N \right)^{-1} 
\left( \bJ^{(i)} \right)^* \br^{(i)} \right];
\end{equation}
\item Use line search to find the step length
\begin{equation}
\alpha^{(i)} = \mathop{\text{argmin}}\limits_{\alpha \in [0, \alpha_{\max}]}
\cF_{\bX} \left( \hq \left( \cdot; \by^{(i-1)} + \alpha \bz^{(i)} \right) \right),
\end{equation}
where $\bX \in \{ \bT, \bS \}$ and we typically take $\alpha_{\max}  = 3$;
\item Update the search potential parameters
\begin{equation}
\by^{(i)} = \by^{(i-1)} + \alpha^{(i)} \bz^{(i)}.
\end{equation}
\end{enumerate}
\EndFor
\State \textbf{Output:} potential estimate
$q^{\text{\scriptsize est}}_{\bX} = \hq \left(\cdot; \by^{(n_{\text{\scriptsize iter}})} \right)$.
\end{algorithmic}
\end{algorithm}

The particular choice of Tikhonov regularization parameter \eqref{eqn:muadap} makes clear
the meaning of the adaptive regularization parameter $\gamma \in (0, 1)$. Since the singular
values of the Jacobian are arranged in the decreasing order, the smaller value of $\gamma$
correspond to more regularization. Typically, one selects $\gamma$ in the range $[0.1, 0.4]$.

\section{Numerical Results}
\label{sec:num}

We provide here the results of numerical experiments for the following ISP setup. 
We consider a 2D ISP in a unit square $\Omega = [0,1] \times [0,1]$ with $m = 8$
extended sources placed at the top boundary 
$\{ \bx = [x_1, 1]^T \;|\; x_1 \in (0,1) \}$, given by
\begin{equation}
p_s(\bx) = \left\{
\begin{tabular}{ll}
1, & if $\quad \dfrac{s-1}{m} + h \leq x_1 \leq \dfrac{s}{m} - h$, \\
0, & otherwise,
\end{tabular}
\right. \quad s = 1,\ldots,m,
\label{eqn:srcnum}
\end{equation}
where the small parameter $h$ is defined below. 
The sampling wavenumbers are
\begin{equation}
k_j = 15 + 5j, \quad j = 1,2,\ldots,n,
\end{equation}
with $n=8$.

The forward problems \eqref{eqn:schrodom}--\eqref{eqn:bc} and 
\eqref{eqn:schrodderiv}--\eqref{eqn:bcderiv} are discretized with finite elements using
Matlab PDE toolbox on a triangular mesh with mesh step $h = 0.03$ resulting in a mesh with 
$5,117$ nodes. The presence of mesh step parameter $h$ in \eqref{eqn:srcnum} guarantees 
that the numerical sources satisfy the non-overlapping condition \eqref{eqn:srcsupp}.

The noiseless synthetic data $\cD$ corresponding to the scattering potential $q^{\text{\scriptsize true}}$
for the numerical experiments is generated by solving \eqref{eqn:schrodom}--\eqref{eqn:bc} 
and \eqref{eqn:schrodderiv}--\eqref{eqn:bcderiv} numerically using the method described above. 
The noise model in the numerical experiments differs slightly from \eqref{eqn:datan}. Explicitly,
instead of adding the noise to the noiseless data $\cD$ to obtain $\cD^{\text{\scriptsize meas}}$
via \eqref{eqn:datan}, the noise is added directly to the blocks $\bb_{ij}$, $\bc_j$, $\bd_j$ and $\partial_k \bd_j$,
$i,j = 1,\ldots,n$, to obtain $\bb_{ij}^{\text{\scriptsize noisy}}$, $\bc_{j}^{\text{\scriptsize noisy}}$,
$\bd_j^{\text{\scriptsize noisy}}$ and $\partial_k \bd_j^{\text{\scriptsize noisy}}$.
The noise added to each entry of the four family of blocks comes from four independent identical normal 
distributions with zero means and standard deviations chosen to satisfy
\begin{equation}
\begin{split}
\frac{\sum\limits_{j=1}^{n} \| \bd_j^{\text{\scriptsize noisy}} - \bd_j \|_F^2}
{\sum\limits_{j=1}^{n} \| \bd_j \|_F^2}  & = 
\frac{\sum\limits_{j=1}^{n} \| \partial_k \bd_j^{\text{\scriptsize noisy}} - \partial_k \bd_j \|_F^2}
{\sum\limits_{j=1}^{n} \| \partial_k \bd_j \|_F^2} \\
& = \frac{\sum\limits_{i \neq j} \| \bb_{ij}^{\text{\scriptsize noisy}} - \bb_{ij} \|_F^2 +
\sum\limits_{j=1}^n \| \bc_{j}^{\text{\scriptsize noisy}} - \bc_{j} \|_F^2}
{\sum\limits_{i \neq j} \| \bb_{ij} \|_F^2 + \sum\limits_{j=1}^{n} \| \bc_{j} \|_F^2} = 
\varepsilon_{\text{\scriptsize noise}}^2,
\end{split}
\end{equation}
where $\varepsilon_{\text{\scriptsize noise}}$ is the desired noise level. For the numerical 
experiments below we take $\varepsilon_{\text{\scriptsize noise}} = 2.5 \cdot 10^{-2}$ 
corresponding to adding $2.5\%$ noise. The blocks are then symmetrized or antisymmetrized, 
as appropriate, to satisfy \eqref{eqn:mat3hermitblock}, \eqref{eqn:cblockasym}, \eqref{eqn:dblocksym} 
and \eqref{eqn:dkdblocksym}. The blocks are then used to compute $\bM^{\text{\scriptsize meas}}$ 
and $\bS^{\text{\scriptsize meas}}$ using formulas \eqref{eqn:ddsij}--\eqref{eqn:ddmjj}.

When implementing the process outline above, it was discovered that due to numerical integration 
errors in computing boundary integrals \eqref{eqn:abblocks}--\eqref{eqn:adkdblocks}, the resulting 
blocks of the stiffness and mass matrices assembled by Algorithm~\ref{alg:ddrom} differ from 
those obtained by (numerical) integration in \eqref{eqn:sij}--\eqref{eqn:mij} even in the absence 
of noise. This introduces a systematic error in ROM construction 
that can be alleviated via error estimation Algorithm~\ref{alg:err} presented below.

\begin{algorithm}
\caption{Boundary integration error estimation}
\label{alg:err}
\begin{algorithmic}
\State \textbf{Input:} reference medium $q_0$.
\State $\bullet$ Solve the forward problems \eqref{eqn:schrodom}--\eqref{eqn:bc} and 
\eqref{eqn:schrodderiv}--\eqref{eqn:bcderiv} numerically for the potential $q_0$
to generate the corresponding boundary data $\cD_0$ and snapshots 
$u^{(s)}_{0, j}$, $\partial_k u^{(s)}_{0, j}$, $j=1,\ldots,n$, $s=1,\ldots,m$.
\State $\bullet$ Use numerical integration on the boundary $\partial \Omega$ to compute the 
blocks \eqref{eqn:abblocks}--\eqref{eqn:adkdblocks} and assemble the corresponding matrices
$\bS_0^{\cD}$ and $\bM_0^{\cD}$ via \eqref{eqn:ddsij}--\eqref{eqn:ddmjj}.
\State $\bullet$ Use numerical integration in the bulk of $\Omega$ to compute the blocks 
\eqref{eqn:sij}--\eqref{eqn:mij} and assemble the corresponding matrices
$\bS_0^{\Omega}$ and $\bM_0^{\Omega}$.
\State $\bullet$ Estimate the boundary integration error for stiffness and mass matrices with
\begin{equation}
\bE_{\bS} = \bS_0^{\cD} - \bS_0^{\Omega}, \text{ and } 
\bE_{\bM} = \bM_0^{\cD} - \bM_0^{\Omega},
\label{eqn:errest}
\end{equation}
respectively.
\State \textbf{Output:} error estimate matrices $\bE_{\bS}$ and $\bE_{\bM}$.
\end{algorithmic}
\end{algorithm}

For the numerical experiments described here it is sufficient to execute Algorithm~\ref{alg:err} for 
$q_0 = 0$ to get accurate enough error estimates \eqref{eqn:errest}. Once the estimates are
obtained, they can be used to correct the mass and stiffness matrices generated by 
Algorithm~\ref{alg:ddrom} by subtracting $\bE_{\bS}$ and $\bE_{\bM}$, respectively.

\begin{figure}
\begin{tabular}{cc}
$q^{\text{\scriptsize 2inc}}$ & $q^{\text{\scriptsize est}}_{\bS}$ \\
\includegraphics[width=0.46\textwidth]{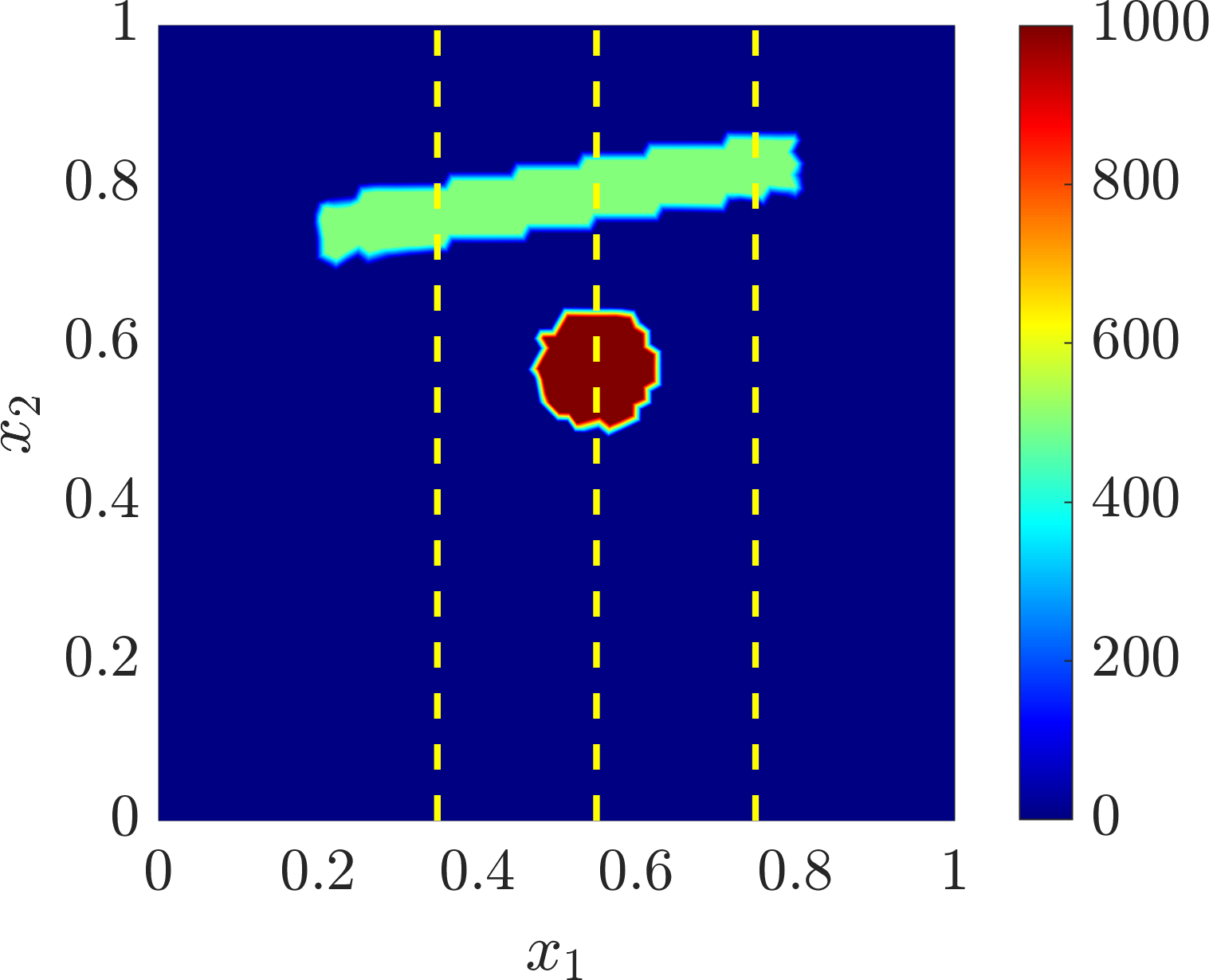} & 
\includegraphics[width=0.46\textwidth]{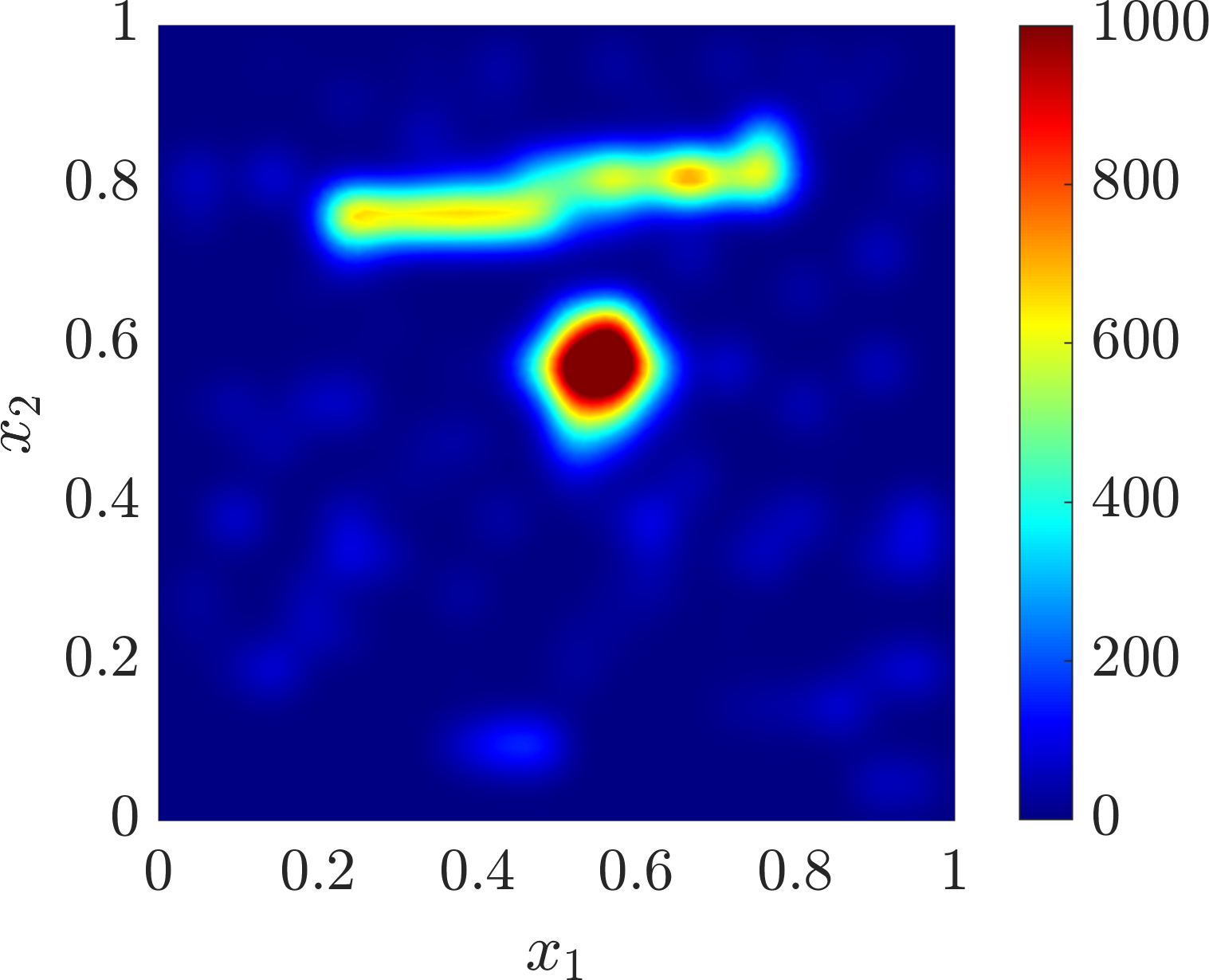} \\
$q^{\text{\scriptsize est}}_{\text{\scriptsize FWI}}$ & $q^{\text{\scriptsize est}}_{\bT}$ \\
\includegraphics[width=0.46\textwidth]{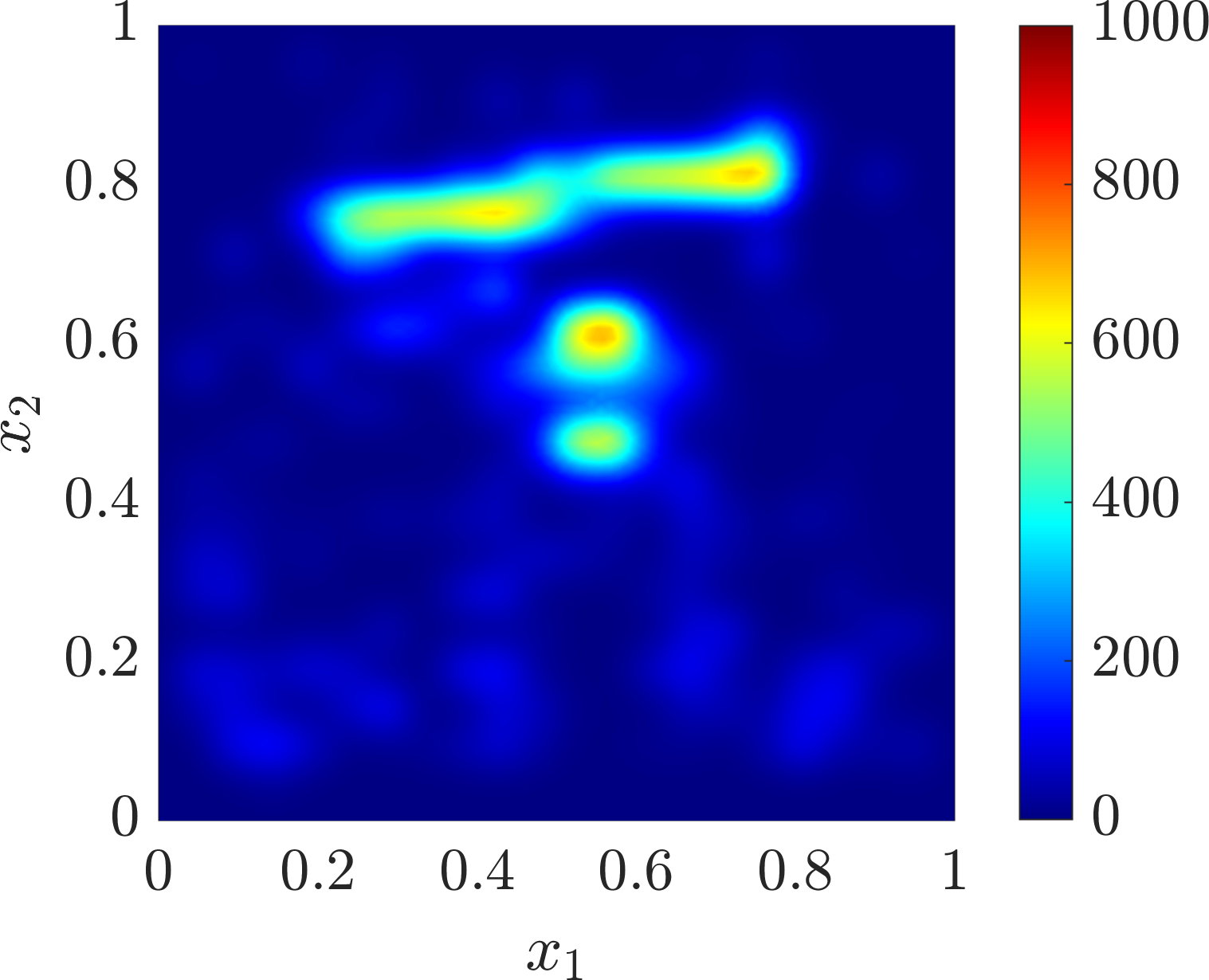} & 
\includegraphics[width=0.46\textwidth]{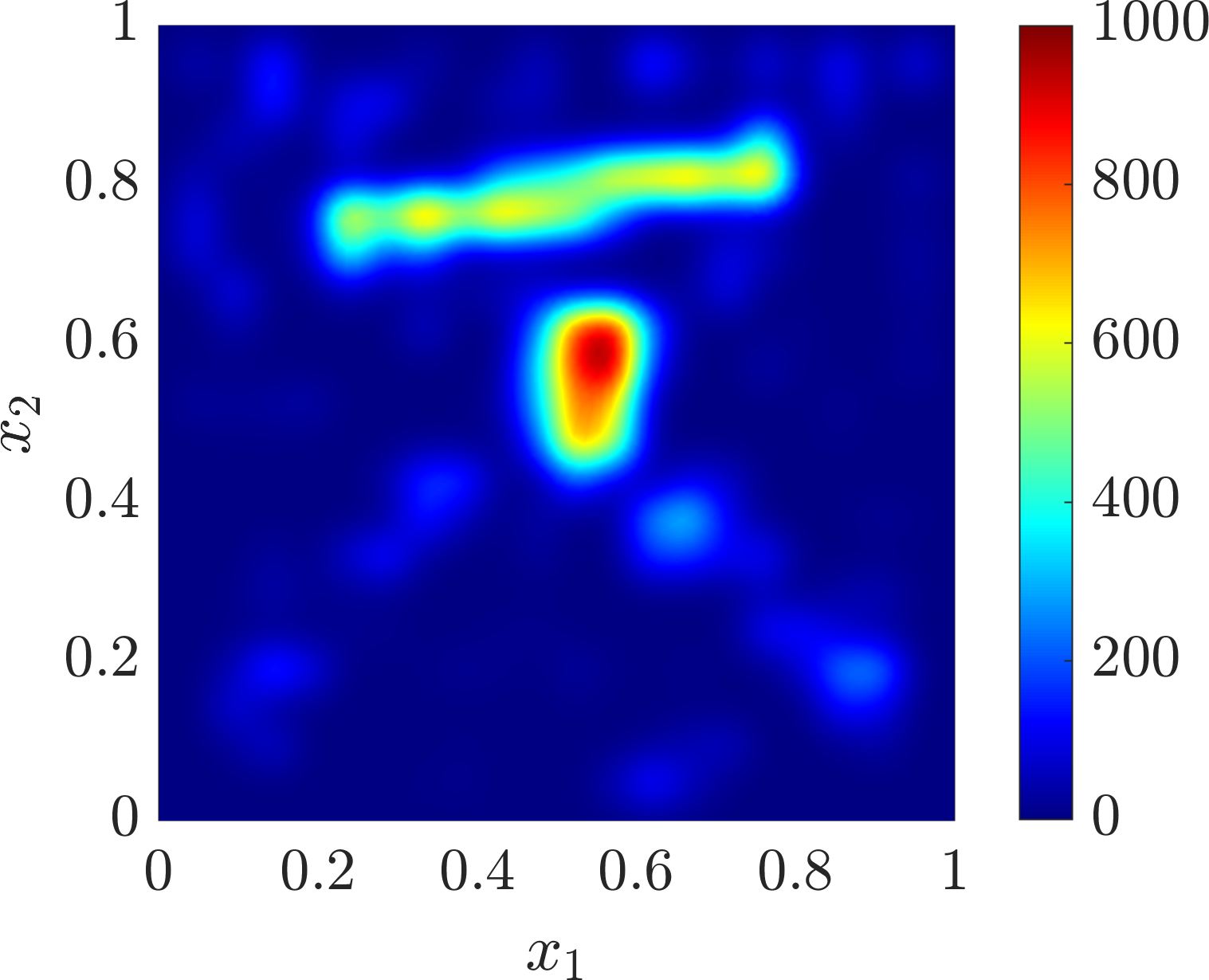} \\
\end{tabular}
\caption{Model used for testing Algorithm~\ref{alg:gn} (top row) and estimated potentials (bottom row).
The dashed yellow lines in top left plot indicate the location of vertical slices shown in 
quality control plots in Figure~\ref{fig:rectschrodoptbslc}.}
\label{fig:rectschrodoptbi}
\end{figure}

To assess the performance of the proposed approach for numerical solution of the ISP we apply 
both variants of Algorithm~\ref{alg:gn} to estimate the potential $q^{\text{\scriptsize 2inc}}$
that contains two inclusions, a slanted weaker extended scatterer and a localized circular scatterer
of higher contrast. The potential $q^{\text{\scriptsize 2inc}}$ is displayed in the top left plot in
Figure~\ref{fig:rectschrodoptbi}. 

To compare the performance of Algorithm~\ref{alg:gn} to the existing methods we implement
a conventional full waveform inversion (FWI) process based on the misfit functional
\begin{equation}
\cF_{\text{\scriptsize FWI}}(\hq) = \sum_{j=1}^{n} \left( 
\left\| \bd_j^{\text{\scriptsize meas}} - \bd_j[\hq] \right\|_F^2 +  
\left\| \partial_k \bd_j^{\text{\scriptsize meas}} - \partial_k  \bd_j[\hq] \right\|_F^2 \right),
\label{eqn:funfwi}
\end{equation}
where the blocks $\bd_j^{\text{\scriptsize meas}}$ and $\partial_k \bd_j^{\text{\scriptsize meas}}$
are computed by Algorithm~\ref{alg:ddrom} from the measured data $\cD^{\text{\scriptsize meas}}$,
while $\bd_j[\hq]$ and $\partial_k  \bd_j[\hq]$ are computed from $\cD[\hq]$ corresponding 
to the search potential $\hq$ using the same algorithm. We apply a regularized Gauss-Newton 
iteration to \eqref{eqn:funfwi} that is essentially Algorithm~\ref{alg:gn} with $\cF_\bX^r$ replaced
with $\cF_{\text{\scriptsize FWI}}$. 

We perform three numerical experiments to compute the estimates of the potential 
$q^{\text{\scriptsize 2inc}}$. First, for the $\bS$-variant of Algorithm~\ref{alg:gn} we use 
noisy data with $\varepsilon_{\text{\scriptsize noise}} = 2.5 \cdot 10^{-2}$ to compute the 
estimate denoted by $q^{\text{\scriptsize est}}_{\bS}$. Second, for the $\bT$-variant of 
Algorithm~\ref{alg:gn} we compute the estimate $q^{\text{\scriptsize est}}_{\bT}$ 
from noiseless data. Third, for the conventional FWI with misfit \eqref{eqn:funfwi} we 
compute the estimate $q^{\text{\scriptsize est}}_{\text{\scriptsize FWI}}$ from the 
noisy data with $\varepsilon_{\text{\scriptsize noise}} = 2.5 \cdot 10^{-2}$.
For all three experiments we set $\gamma = 0.2$. We perform $n_{\text{\scriptsize iter}} = 10$ 
iterations for the first and third experiments, while taking $n_{\text{\scriptsize iter}} = 20$
for the second. The search space $\cQ$ is taken to contain $N = 20 \times 20 = 400$ Gaussian 
basis functions with peaks located on a uniform $20 \times 20$ rectangular grid in $\Omega$.

We display in Figure~\ref{fig:rectschrodoptbi} the estimated potentials 
$q^{\text{\scriptsize est}}_{\bS}$, $q^{\text{\scriptsize est}}_{\bT}$ and 
$q^{\text{\scriptsize est}}_{\text{\scriptsize FWI}}$. We observe first that 
$q^{\text{\scriptsize est}}_{\bS}$ provides the best estimate of the potential. It captures 
very well positions, shapes and magnitudes of both scatterers while having negligibly small
artifacts even though noisy data is used. \hl{The estimate $q^{\text{\scriptsize est}}_{\bT}$
delivers the second best reconstruction capturing well both the locations and magnitudes of 
both scatterers, but introducing some artifacts in the estimate and somewhat distorting 
the shape of the circular scatterer.} The FWI estimate 
$q^{\text{\scriptsize est}}_{\text{\scriptsize FWI}}$ recovers the topmost scatterer well,
but does a poor job with the circular one missing its contrast and also splitting into two 
scatterers, possibly due to the effect of multiple reflections.

\begin{figure}
\begin{tabular}{ccc}
$x_1 = 0.35$ & $x_1 = 0.55$ & $x_1 = 0.75$ \\
\includegraphics[width=0.305\textwidth]{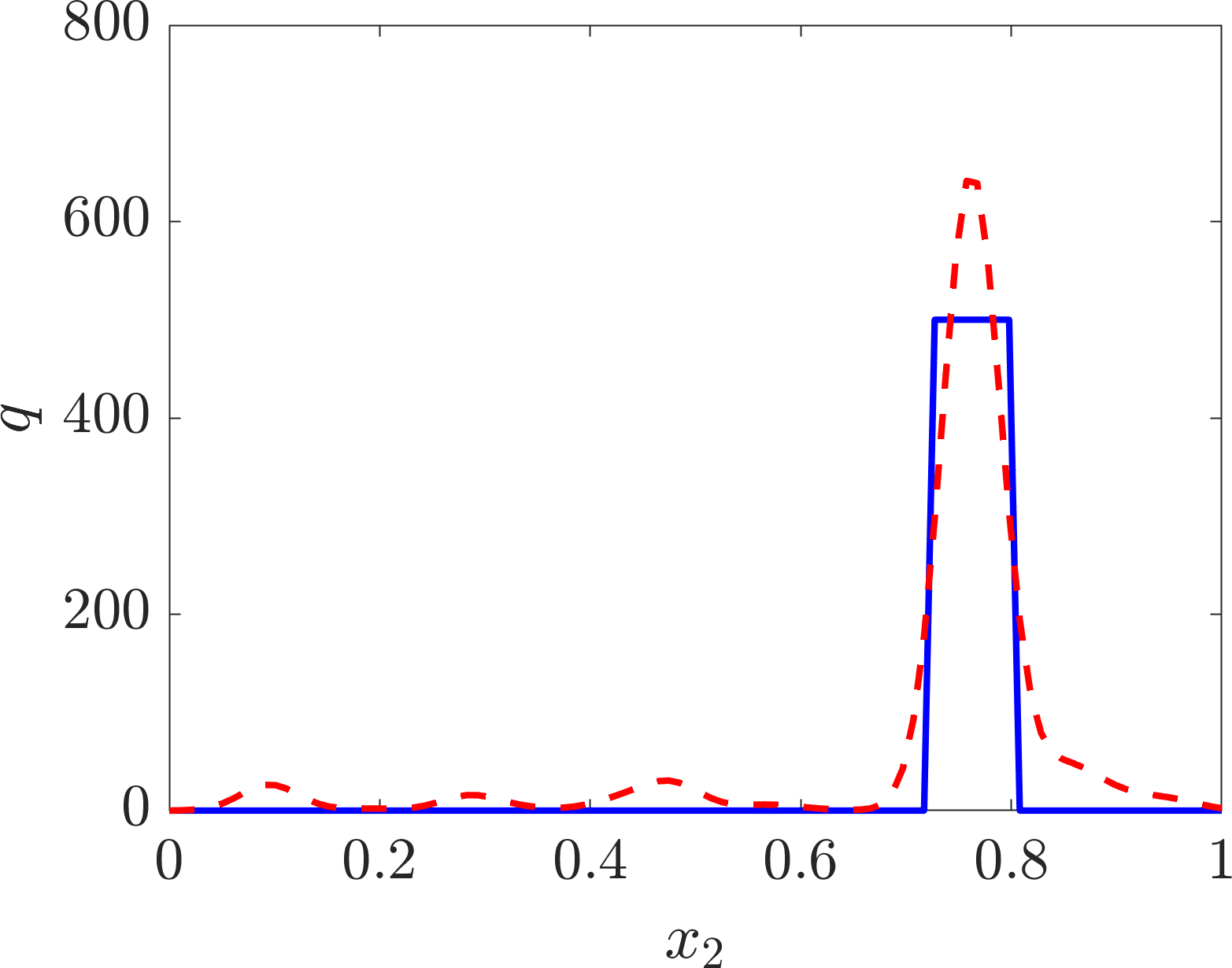} & 
\includegraphics[width=0.305\textwidth]{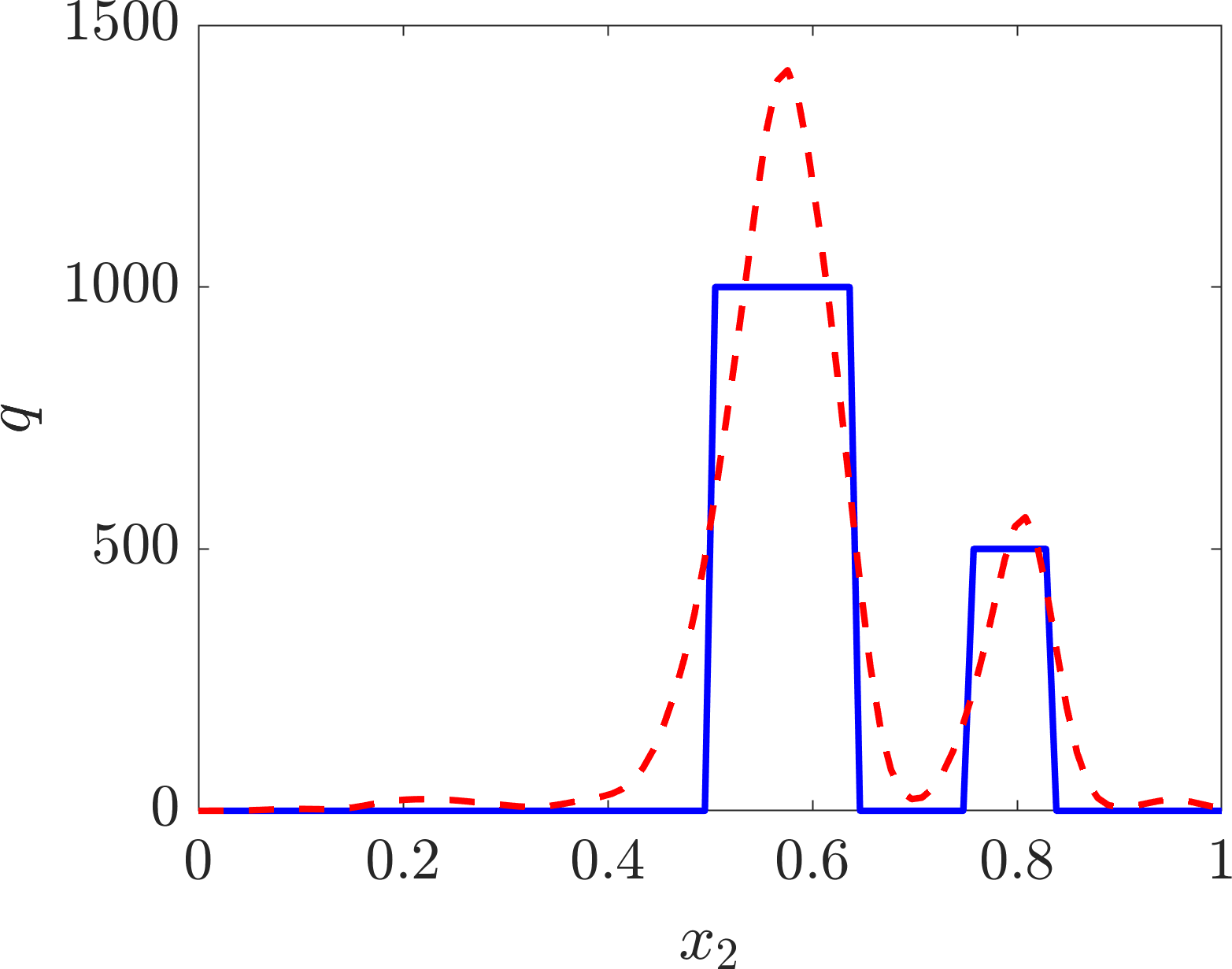} & 
\includegraphics[width=0.305\textwidth]{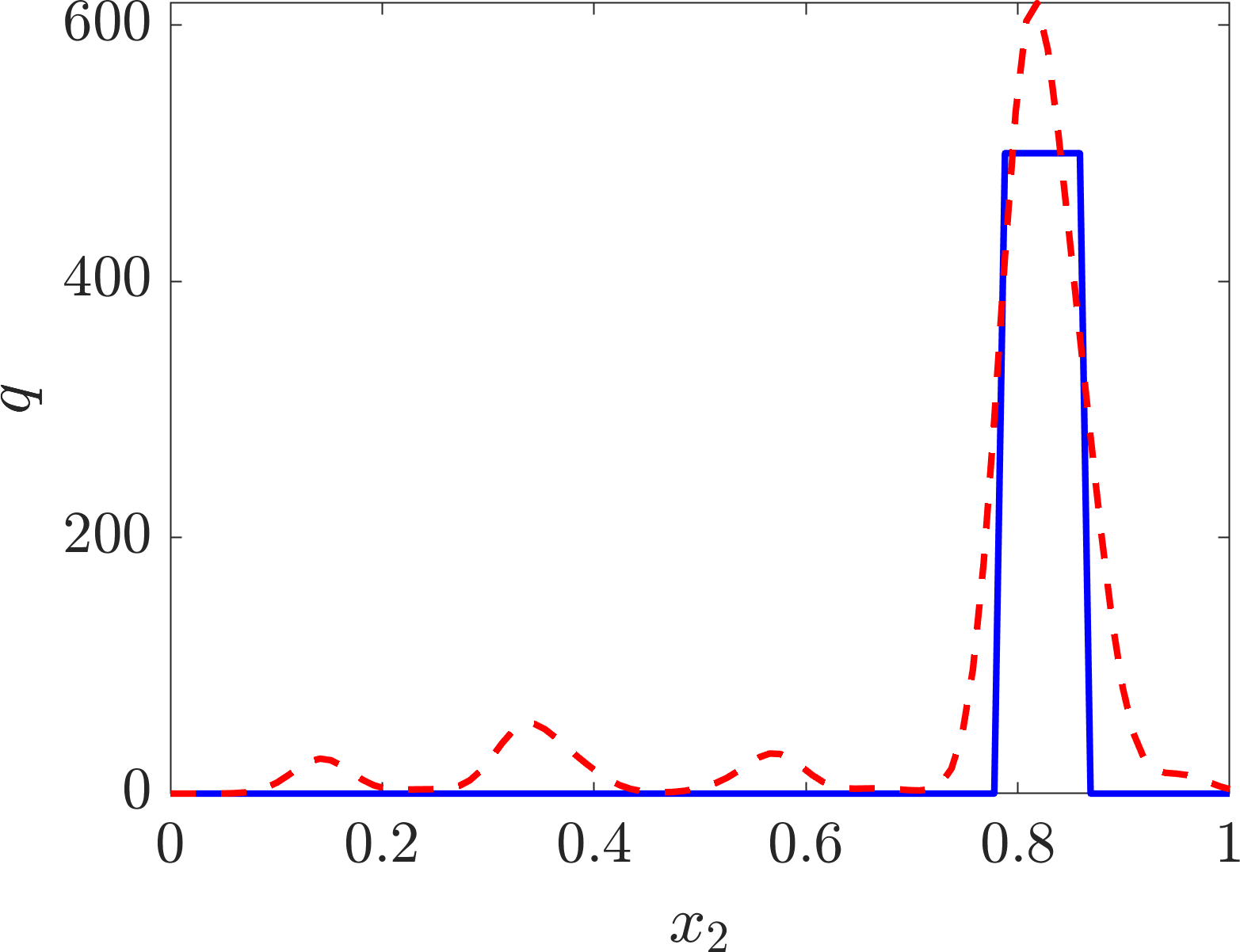} \\
\includegraphics[width=0.305\textwidth]{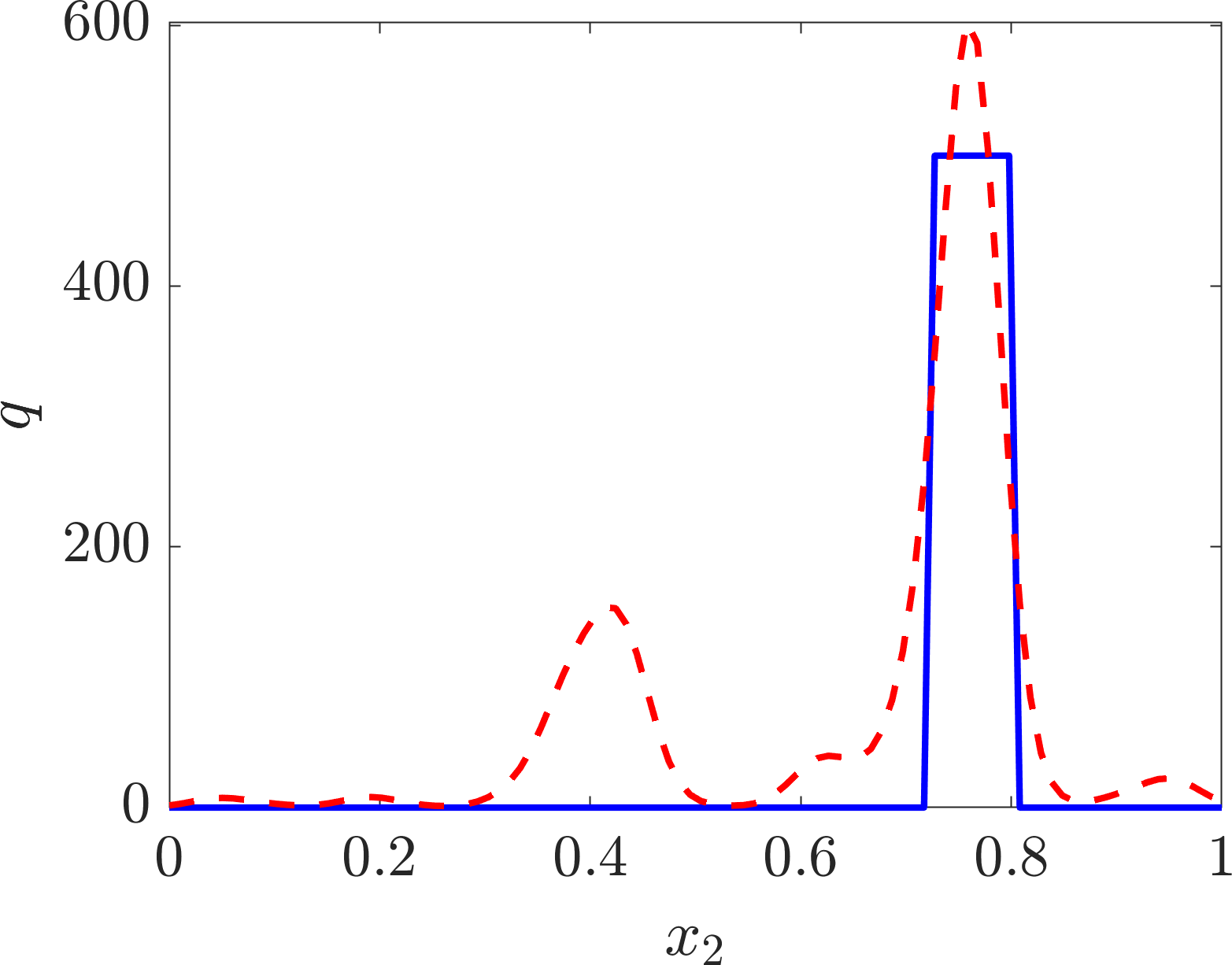} & 
\includegraphics[width=0.305\textwidth]{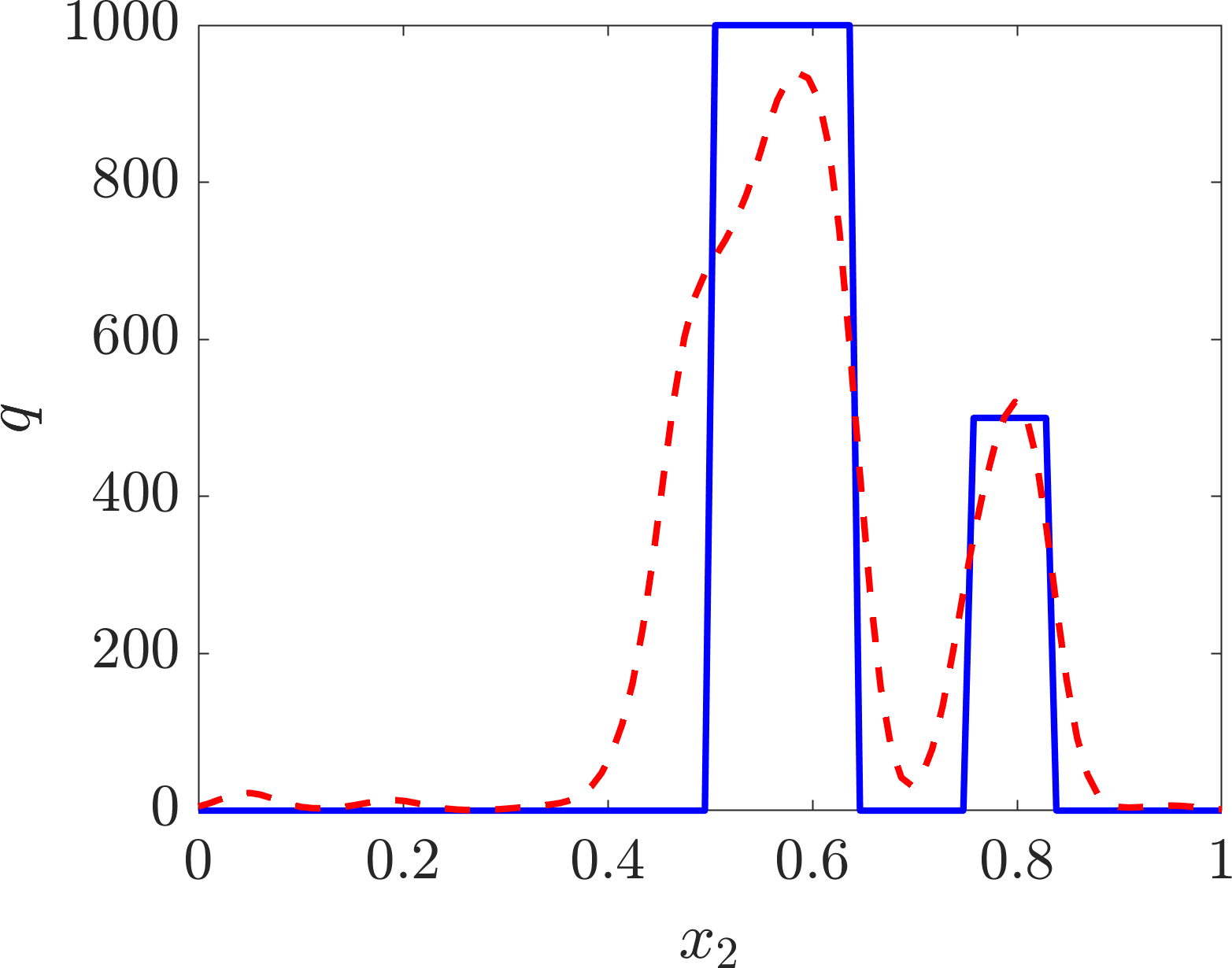} & 
\includegraphics[width=0.305\textwidth]{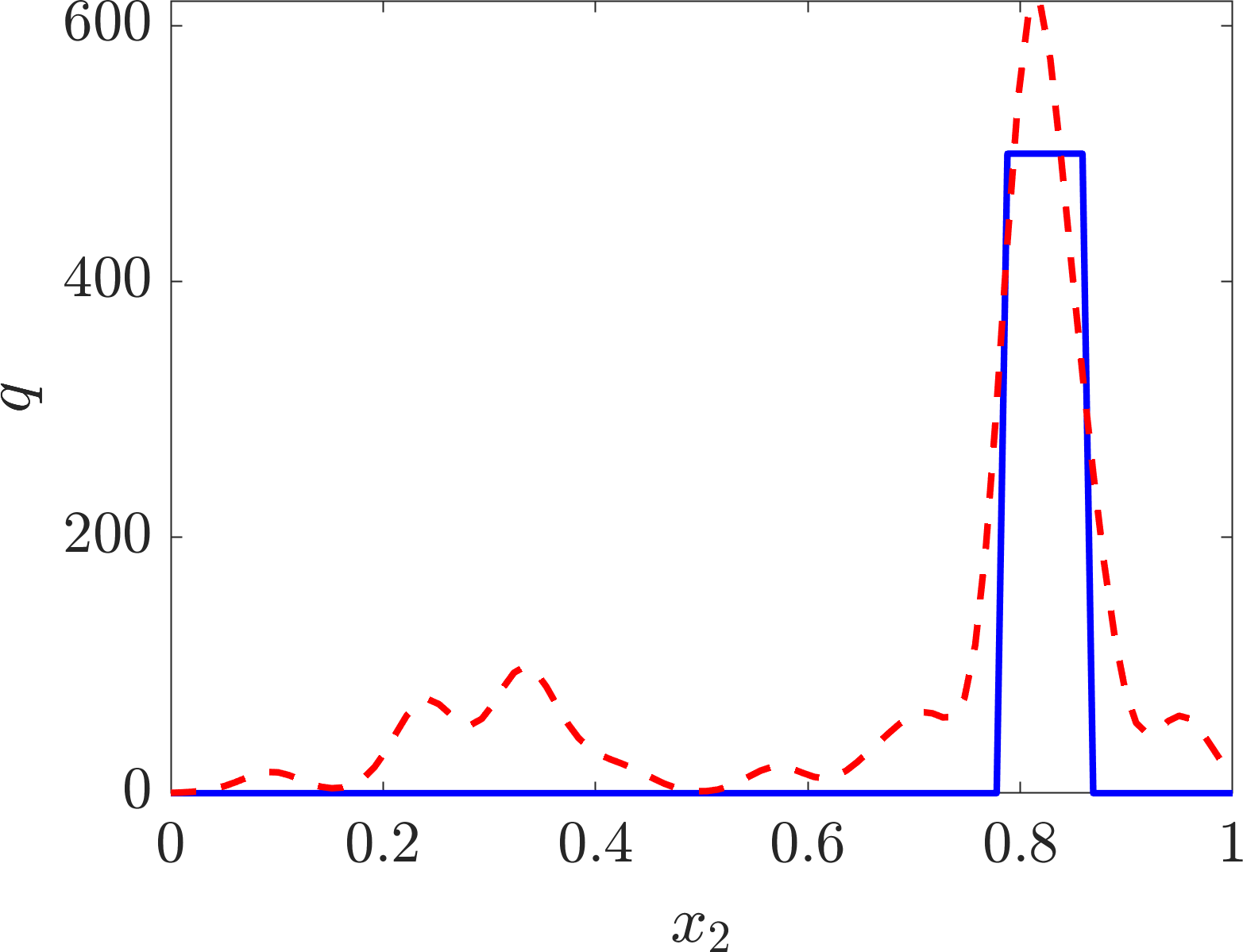} \\
\includegraphics[width=0.305\textwidth]{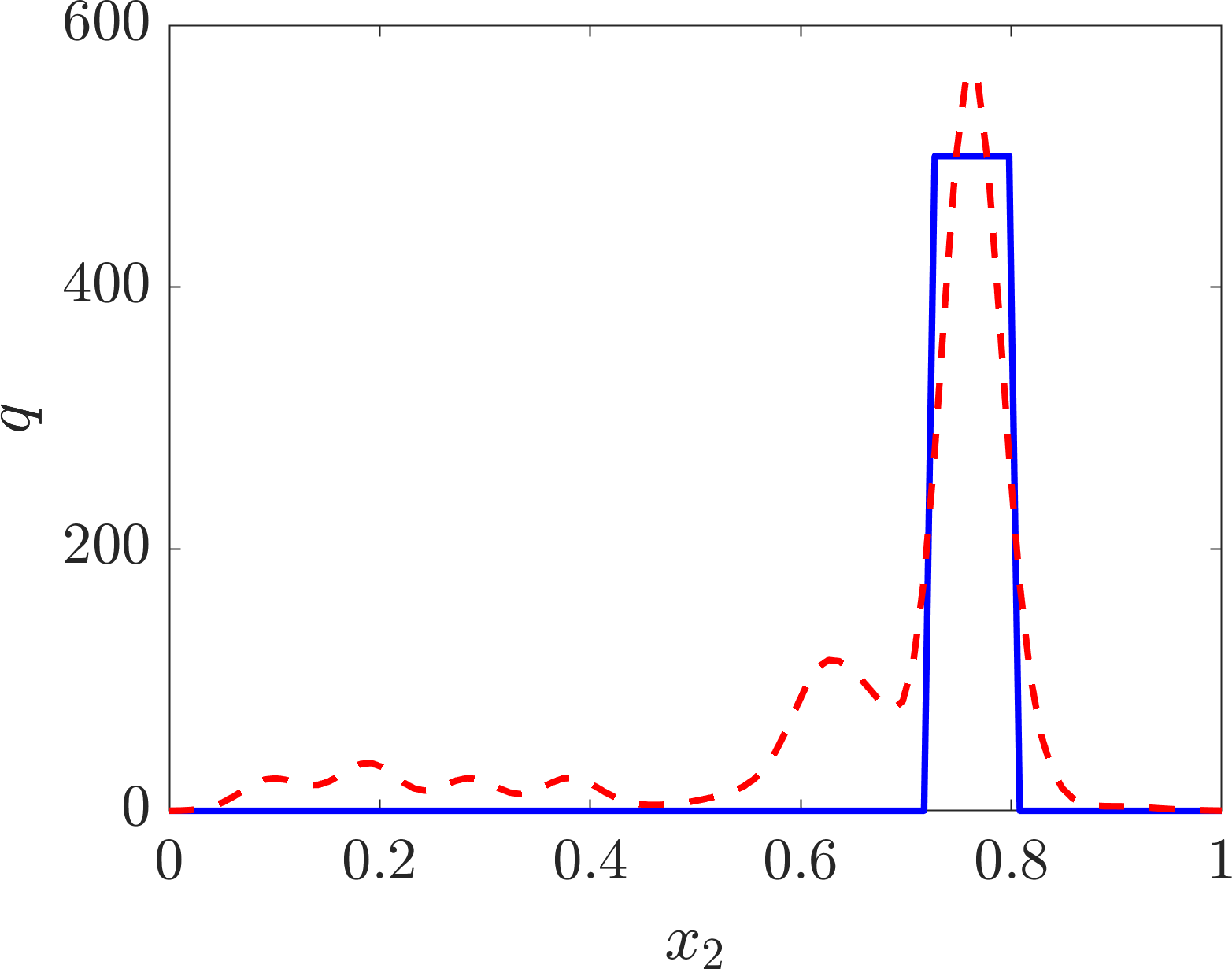} & 
\includegraphics[width=0.305\textwidth]{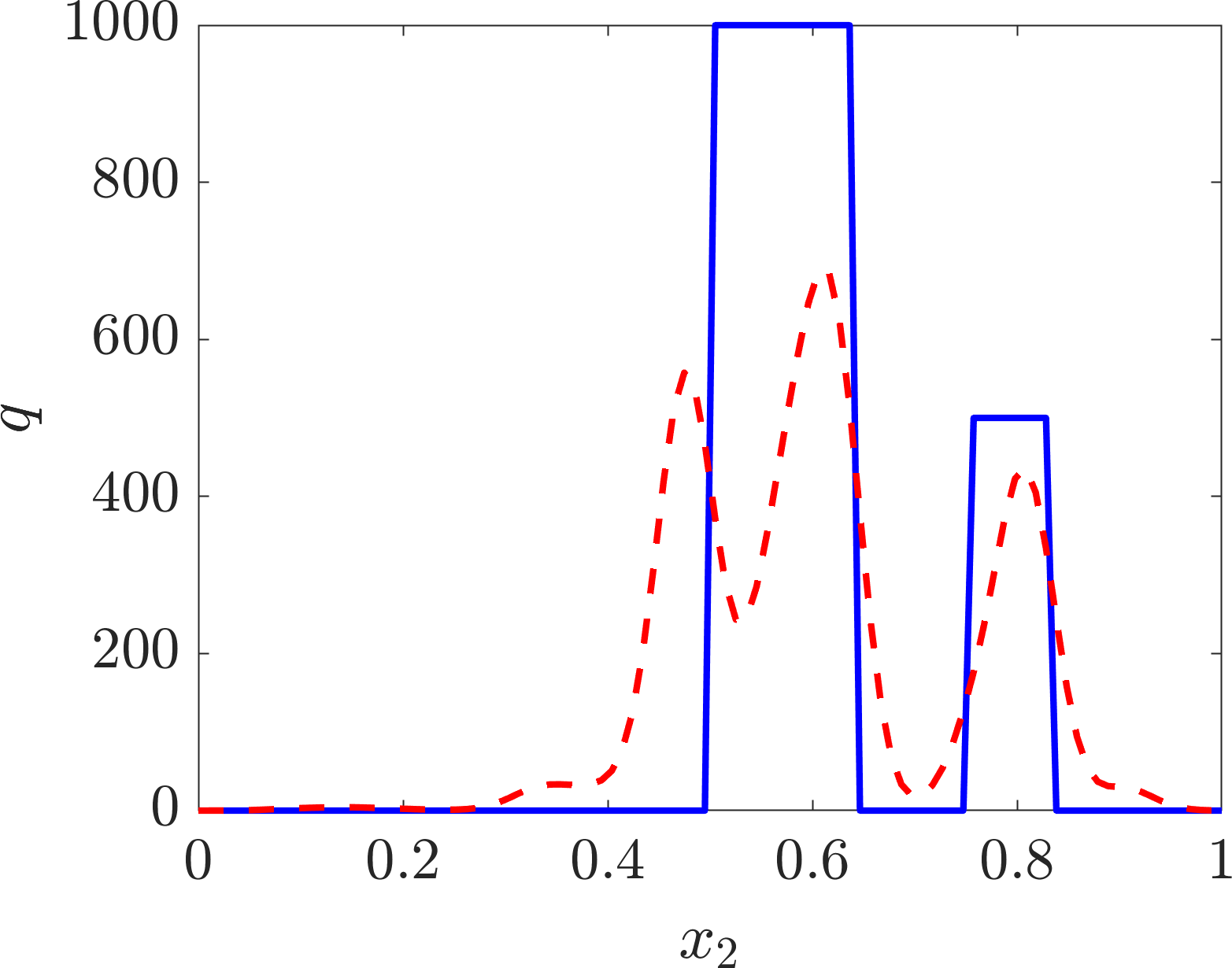} & 
\includegraphics[width=0.305\textwidth]{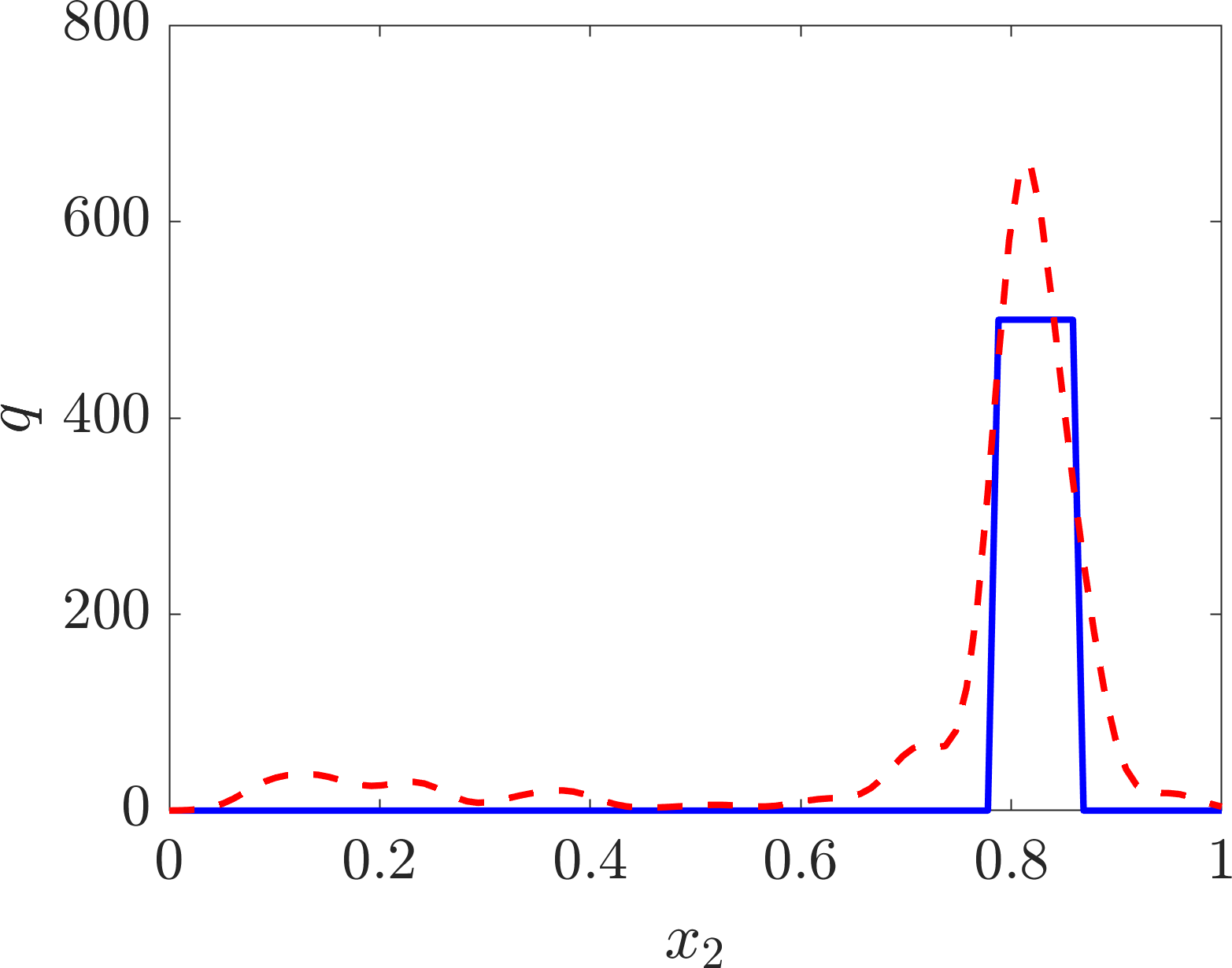} 
\end{tabular}
\caption{Potential estimate quality control: vertical slices of the target potential 
$q^{\text{\scriptsize 2inc}}$ (solid blue lines) and its estimates (dashed red lines)
for three different values of $x_1$. Top row: $q^{\text{\scriptsize est}}_{\bS}$;
middle row: $q^{\text{\scriptsize est}}_{\bT}$;
bottom row: $q^{\text{\scriptsize est}}_{\text{\scriptsize FWI}}$.}
\label{fig:rectschrodoptbslc}
\end{figure}

In order to provide quality control of potential estimates, we show in 
Figure~\ref{fig:rectschrodoptbslc} vertical slices of both the target potential 
$q^{\text{\scriptsize 2inc}}$ and all three estimates at three different locations shown 
in the top left plot in Figure~\ref{fig:rectschrodoptbi} as dashed yellow lines. Confirming our 
conclusions above, the vertical slices show an excellent agreement between $q^{\text{\scriptsize 2inc}}$
and $q^{\text{\scriptsize est}}_{\bS}$ while emphasizing some deficiencies of 
$q^{\text{\scriptsize est}}_{\bT}$ and $q^{\text{\scriptsize est}}_{\text{\scriptsize FWI}}$.
In particular, the middle leftmost and rightmost plots shows the artifacts affecting 
$q^{\text{\scriptsize est}}_{\bT}$, while the bottom middle plot points to the poor recovery of 
the circular scatterer in $q^{\text{\scriptsize est}}_{\text{\scriptsize FWI}}$.

Overall, the $\bS$-variant of Algorithm~\ref{alg:gn} provides the best estimate 
$q^{\text{\scriptsize est}}_{\bS}$. This leads to a question of why should the $\bT$-variant 
be considered if it provides somewhat worse results numerically. This aspect is discussed in the next
section.

\section{Conclusions and future work}
\label{sec:conclude}

We successfully extended the approaches of \cite{tataris2023reduced,TristanAndreas23ROM} for
solving numerically the inverse scattering problem for the scattering potential of Schr\"{o}dinger  
using data-driven ROM to the cases of two and more spatial dimensions. We obtained the expressions
for computing the mass and stiffness matrices of Schr\"{o}dinger problem in Galerkin framework
from the knowledge of frequency domain data measured on the boundary of the domain of interest.
Algorithm~\ref{alg:gn} for estimating the scattering potential was developed, implemented and
tested numerically displaying a potential for high quality estimates that outperform the conventional
frequency domain FWI both in terms of contrast and spatial accuracy. Algorithm~\ref{alg:gn} 
admits two variants ($\bS$ and $\bT$) depending on which ROM matrix is fitted during the 
optimization process. While the numerical studies in Section~\ref{sec:num} suggest the advantage 
of the $\bS$-variant, we believe that having two variants of Algorithm~\ref{alg:gn} is beneficial 
for future research, as discussed below.

The obvious and important next step in ROM-based inversion for coefficients of wave PDEs in 
the frequency domain is the extension of the techniques presented here from Schr\"{o}dinger
equation to Helmholtz equation. The main difference between the two is that the presence of 
inclusions in the medium has both the kinematic effect (changes in travel times) and scatters 
the probing waves in the Helmholtz case, while the kinematics in the Schr\"{o}dinger case is fixed, 
i.e., the speed of propagation is constant. This leads to various kinds of complications for estimating
the coefficient of the Helmholtz equation, including the effects like cycle-skipping that lead to 
severe non-convexity of conventional optimization formulations like \eqref{eqn:lsdata}
or \eqref{eqn:funfwi}. The studies of time-domain ROM-based inversion in 
\cite{borcea2020reduced,borcea2023waveform1,druskin2016direct,druskin2018nonlinear} 
suggest that in the presence of kinematic effects, the $\bT$-variant of a method similar to 
Algorithm~\ref{alg:gn}, i.e., minimization of an objective like \eqref{eqn:funtr}, should have 
an advantage over minimizing \eqref{eqn:funsr}. Studying this question as well as adapting the 
techniques presented here to the Helmholtz case remains our immediate research priority.

%

\appendix
\section{Proofs}
\label{app:proofs}

\begin{proof}[Proof of Proposition \ref{prop:mij}]
From the weak form \eqref{eqn:weaksom} with $k=k_i$ we obtain
\begin{align}
    \int_\Omega \overline{\nabla u^{(r)}_i} \cdot \nabla {v} \; d \bx + 
\int_{\Omega} q \overline{u^{(r)}_i} v d \bx - 
k^2_i \int_\Omega \overline{u^{(r)}_i} v d \bx + 
\imath k_i \int_{\partial \Omega} \overline{u^{(r)}_i} v d\Sigma = 
\int_{\partial \Omega} \overline{p_r} v d\Sigma,
\end{align}
for $r = 1,\ldots,m$.
Taking $v=u^{(s)}_j$, $s = 1,\ldots,m$, as a test function for $i \neq j$, we use 
\eqref{eqn:sij}--\eqref{eqn:mij}, \eqref{eqn:dj} and \eqref{eqn:bblocks} to get
\begin{align}
[\bs_{ij}]_{rs}-k^2_i  [\bm_{ij}]_{rs}+\imath  k_i [\bb_{ij}]_{rs}= [\bd_j]_{rs}.
\label{eqn:App-romm1}
\end{align}
Similarly, for $k = k_j$ we have
\begin{align}
[\bs_{ji}]_{sr}-k^2_j  [\bm_{ji}]_{sr}+\imath  k_j [\bb_{ji}]_{sr}= [\bd_i]_{sr}.
\label{eqn:App-romm1a}
\end{align}
Taking the complex conjugate of \eqref{eqn:App-romm1a} and using \eqref{eqn:mat3hermitblock} gives
\begin{align}
[\bs_{ij}]_{rs}-k^2_j  [\bm_{ij}]_{rs}-\imath  k_j [\bb_{ij}]_{rs}= \overline{[\bd_i]}_{sr}.
\label{eqn:App-romm2}
\end{align}
Subtracting \eqref{eqn:App-romm2} from \eqref{eqn:App-romm1} we obtain
\begin{align}
- k^2_i  [\bm_{ij}]_{rs}+k^2_j  [\bm_{ij}]_{rs}+\imath  k_i [\bb_{ij}]_{rs}+\imath  k_j [\bb_{ij}]_{rs}
= [\bd_j]_{rs}-\overline{[\bd_i]}_{sr},
\end{align}
implying
\begin{align}
(k^2_j - k^2_i) [\bm_{ij}]_{rs} + \imath ( k_i + k_j ) [\bb_{ij}]_{rs} 
= [\bd_j]_{rs} - \overline{[\bd_i]}_{sr}.
\end{align}
Therefore,
\begin{align}
[\bm_{ij}]_{rs} = \frac{\overline{[\bd_{i}]}_{sr} - [\bd_{j}]_{rs}}{k^2_i - k^2_j} 
- \imath \frac{[\bb_{ij}]_{rs} }{k_j - k_i}, \quad r,s=1,...,m, \quad i \neq j.
\label{eqn:mijrs}
\end{align}
Since the matrices $\bd_{i}$ are complex-symmetric according to Propositon~\ref{prop: complex symmetric data}, 
we have
\begin{align}
\overline{[\bd_{i}]}_{sr} = \overline{[\bd_{i}]}_{rs}, \quad r, s = 1,\ldots,m, 
\end{align}
hence
\begin{align}
\overline{\bd_{i}} = \overline{\bd_{i}}^T=\bd_{i}^*, \quad i = 1, \ldots, n.
\label{eqn:bdis}
\end{align}
Thus, \eqref{eqn:mijrs} becomes
\begin{align}
\bm_{ij} = \frac{\bd_{i}^* - \bd_{j}}{k^2_i - k^2_j} 
- \imath \frac{\bb_{ij} }{k_j - k_i}, \quad i \neq j.
\end{align}
\end{proof}

\begin{proof}[Proof of Proposition \ref{prop:mjj}]
For the diagonal blocks of the mass matrix we first write
\begin{align}
\bm_{ii}=\lim_{k_j \to k_i}\frac{\bm_{ij}+\bm_{ji}}{2}.
\label{eqn:miilim}
\end{align}
We also split the blocks of the mass matrix into $\bm_{ij}=\bm_{ij}^{(1)} + \bm_{ij}^{(2)}$, where
\begin{align}
\bm_{ij}^{(1)} = \frac{\bd_{i}^* - \bd_{j}}{k^2_i - k^2_j}, 
\label{eqn:mij1}
\end{align}
and
\begin{align}
\bm_{ij}^{(2)} =  -\imath \frac{\bb_{ij}}{k_j-k_i}.
\label{eqn:mij2}
\end{align}
To take the limit \eqref{eqn:miilim} we first consider the contribution of \eqref{eqn:mij1}. 
Using \eqref{eqn:bdis}, we write
\begin{align}
    \left[\bm_{ij}^{(1)} \right]_{rs} + \left[ \bm_{ji}^{(1)} \right]_{rs} & = 
    \frac{\overline{[\bd_{i}]}_{rs} - [\bd_{j}]_{rs}}{k^2_i - k^2_j} +  
    \frac{\overline{[\bd_{j}]}_{rs} - [\bd_{i}]_{rs}}{k^2_j - k^2_i} \\
    & = 
    \frac{[\bd_{j}]_{rs}-\overline{[\bd_{i}]}_{rs}+\overline{[\bd_{j}]}_{rs} - [\bd_{i}]_{rs}}{k^2_j - k^2_i} \\
 & =  \frac{1}{k_j+k_i}\frac{[\bd_{j}]_{rs}- [\bd_{i}]_{rs}
      +\overline{[\bd_{j}]}_{rs}-\overline{[\bd_{i}]}_{rs} }{k_j - k_i}.
\end{align}
As $k_j\to k_i$, we observe that
\begin{align}
2 \bm_{ii}^{(1)} = 
\frac{1}{2k_i}  \Big( \partial_k \bd_{i} + \overline{\partial_k \bd_{i}} \Big) = 
\frac{1}{k_i}\Re \big( \partial_k \bd_i \big), \quad i = 1,\ldots,n.
  \label{eqn:mii1}
\end{align}
For the contribution of \eqref{eqn:mij2} we have
\begin{align}   
\left[\bm_{ij}^{(2)} \right]_{rs} + \left[ \bm_{ji}^{(2)} \right]_{rs}  = 
- \imath \frac{[\bb_{ij}]_{rs}}{k_j-k_i}- \imath \frac{[\bb_{ji}]_{rs}}{k_i-k_j}.
\end{align}
Since $[\bb_{ij}]_{rs}=\int\limits_{\partial \Omega} \overline{u^r(\bx;k_i}){u^s(\bx;k_j )} d\Sigma$, we write 
\begin{align*}
- \left[\bm_{ij}^{(2)} \right]_{rs} - \left[ \bm_{ji}^{(2)} \right]_{rs} = 
   \frac{\imath}{k_j-k_i} \int\limits_{\partial \Omega} \overline{u^r(\bx;k_i)}{u^s(\bx; k_j )}d\Sigma +
   \frac{\imath}{k_i-k_j} \int\limits_{\partial \Omega}\overline{u^r(\bx; k_j)}{u^s(\bx; k_i)}d\Sigma = \\
     =  \frac{\imath}{k_j-k_i} \int\limits_{\partial \Omega}
     \Big( \overline{u^r(\bx;k_i)}\{{u^s(\bx;k_j)}-{u^s(\bx;k_i )}\} 
     - {u^s(\bx;k_i)}\{ \overline{u^r(\bx;k_j)-u^r(\bx;k_i)}\} \Big) d\Sigma.
\end{align*}
As $k_j\to k_i$, we obtain
\begin{align}
    -\lim_{j\to i} 
    \left( \left[\bm_{ij}^{(2)} \right]_{rs} + \left[ \bm_{ji}^{(2)} \right]_{rs} \right)=
    \imath \int_{\partial \Omega}\Big( 
\overline{u^r(\bx;k_i) }{\partial_k u^s(\bx;k_i)}-{u^s(\bx;k_i)} \overline{\partial_k u^r(\bx;k_i)}  \Big)d\Sigma,
\end{align}
hence
\begin{align}
     2 \left[\bm_{ii}^{(2)} \right]_{rs} =  \imath \int_{\partial \Omega}\Big( 
-\overline{u^r(\bx;k_i) }{\partial_ku^s(\bx;k_i)}+{u^s(k_i,x)} \overline{\partial_k u^r(\bx;k_i)}  \Big)d\Sigma,
\end{align}
which gives
\begin{align}
     2 \bm_{ii}^{(2)} = \imath \bc_{i}, \quad i = 1,\ldots,n.
     \label{eqn:mii2}
\end{align}
Combining \eqref{eqn:mii1} and \eqref{eqn:mii2}, we obtain that
\begin{align}
    \bm_{ii} = \frac{1}{2k_i} \Re \big( \partial_k \bd_i \big) +
    \frac{\imath}{2} \bc_{i}, \quad i=1,...,n.
\end{align}

\end{proof}
\begin{proof}[Proof of Proposition \ref{prop:sij}]
We recall from \eqref{eqn:App-romm1} for $i\neq j$ that
\begin{align}
[\bs_{ij}]_{rs}-k^2_i  [\bm_{ij}]_{rs}+\imath  k_i [\bb_{ij}]_{rs}= [\bd_j]_{rs},
\end{align}
which we multiply by $k_j^2$ to obtain
\begin{align}
k^2_j [\bs_{ij}]_{rs}-k^2_j k^2_i  [\bm_{ij}]_{rs}+\imath  k^2_j k_i [\bb_{ij}]_{rs}= k^2_j [\bd_j]_{rs}.
\label{eqn:App-romm3}
\end{align}
Similarly,
\begin{align}
k^2_i [\bs_{ij}]_{rs}-k^2_i k^2_j  [\bm_{ij}]_{rs}-\imath k^2_i k_j [\bb_{ij}]_{rs}= k^2_i \overline{[\bd_i]}_{sr}.
\label{eqn:App-romm4}
\end{align}
Subtracting \eqref{eqn:App-romm4} from \eqref{eqn:App-romm3} we have
\begin{align}
    k^2_j [\bs_{ij}]_{rs}-k^2_i [\bs_{ij}]_{rs}+\imath  k^2_j k_i [\bb_{ij}]_{rs}+\imath k^2_i k_j [\bb_{ij}]_{rs}= k^2_j [\bd_j]_{rs}- k^2_i \overline{[\bd_i]}_{sr},
\end{align}
implying
\begin{align}
[\bs_{ij}]_{rs} & = \frac{k^2_j [\bd_j]_{rs}- k^2_i \overline{[\bd_i]}_{sr}}{k^2_j-k^2_i} 
- \frac{\imath  k^2_j k_i [\bb_{ij}]_{rs}+\imath k^2_i k_j [\bb_{ij}]_{rs}}{k^2_j-k^2_i} \\
& = \frac{ k^2_i \overline{[\bd_i]}_{sr}-k^2_j [\bd_j]_{rs}}{k^2_i-k^2_j} 
- \imath\frac{  k^2_j k_i + k^2_i k_j }{k^2_j-k^2_i}[\bb_{ij}]_{rs}, \quad r,s = 1,\ldots,m, \ i \neq j.
\end{align}
Using \eqref{eqn:bdis} yields the desired result.
\end{proof}
\begin{proof}[Proof of Proposition \ref{prop:sjj}]
Introducing $\bff_j = k^2_j \mathbf d_j$, $j = 1,\ldots,n$, the entries of the stiffness matrix blocks 
can be written as
\begin{align}
    [\bs_{ij}]_{rs}= \frac{ \overline{[\mathbf \bff_i]}_{rs}- [\mathbf \bff_j]_{rs}}{k^2_i-k^2_j}-\imath\frac{k^2_j k_i + k^2_i k_j}{k_i+k_j}\frac{  [\bb_{ij}]_{rs}}{k_i-k_j}, \quad r,s=1,\ldots,m,
\end{align}
for $i \neq j$. 
Similarly to the proof of Proposition~\ref{prop:mjj}, we can obtain that
\begin{equation}
\begin{split}
2 [\bs_{ii}]_{rs} = & 
\frac{1}{k_i}\Big([\partial_k \mathbf  \bff_i]_{rs}+ \overline{[\partial_k \mathbf  \bff_i]}_{rs} \Big) + \\
& \frac{\imath 2k_i^{3}}{2k_i}\int_{\partial \Omega}\Big( 
-\overline{u^r(\bx;k_i) } {\partial_k u^s(\bx;k_i)}+{u^s(\bx;k_i)} \overline{\partial_k u^r(\bx;k_i)}  \Big)d\Sigma,
\end{split}
\end{equation}
hence
\begin{equation}
\begin{split}
   2 [\bs_{ii}]_{rs} = &\frac{1}{2k_i} \Big( k_i^2 [\partial_k \bd_{i}]_{rs} +k^2_i[ \bd_{i}]_{rs}+ 
    k_i^2 \overline{[\partial_k \bd_{i}]}_{sr} + k^2_i \overline{[ \bd_{i}]}_{sr} \Big)+\\
    & \imath k_i^{2}\int_{\partial \Omega}\Big( - \overline{u^r(\bx;k_i) } {\partial_k u^s(\bx;k_i)} + 
    {u^s(\bx;k_i)} \overline{\partial_k u^r(\bx;k_i)}  \Big) d\Sigma,
\end{split}
\end{equation}
for $i = 1,\ldots,n$.
Simplifying and using the definition \eqref{eqn:cblocks} we find
\begin{align} 
    \bs_{ii} = \frac{1}{2} \Big( k_i \Re ( \partial_k \bd_i )+ 2 \Re(\bd_i) \Big) 
+ \frac{\imath k_i^2}{2} \bc_i, \quad i=1,\ldots,n.
\end{align}
\end{proof}

\section{Proof of Theorem~\ref{thm:uderiv}}
\label{app:diff}
Before proving the desired result, we recall the following variant of the implicit function theorem, 
see, e.g., \cite{Connolly_1990,168}.
\begin{theorem}
Consider $F: \R \times H \to W$, where $H$ and $W$ are Banach spaces. 
Assume that there exists an open set $I \subset \R$ such that for every 
$k \in I$ there exists a unique $u = u(k)\in H$ such that  
\begin{align}
F(k, u)=0. 
\label{B1 Appendinx}
\end{align}
Then if 
\begin{align}
F : \R \times H\to W
\end{align} 
is continuous, and
\begin{align}
\partial_u F : \R \times H\to W
\end{align} 
is continuous, and
\begin{align}
\fa k \in I: \; (\partial_u F(k,u))^{-1} : W \to H
\end{align} 
exists and is bounded, then there exists a continuous map such that
\begin{align}
I \ni k\mapsto u(k) \in H.
\end{align}
Also, if $\partial_k F$ is continuous, we obtain that $u$ is Frech\'{e}t-differentiable 
and its derivative at $k_0$ is given by
\begin{align}
[\partial_k u](k_0) = - \left\{ \partial_u F(k_0,u(k_0)) \right\}^{-1} \partial_k F(k_0,u(k_0)).
\label{eqn:dkuk0}
\end{align}
\end{theorem}

To prove Theorem~\ref{thm:uderiv} we take 
\begin{align}
F^{(s)}: \; (k,u) \in \R_+ \times H^1(\Omega) \to {H^1(\Omega)'}, \quad s = 1,\ldots,m,
\end{align}
defined by \eqref{eqn:opf}. Observe that for \hl{every open} interval $I\subset \R_+$, 
relation \eqref{B1 Appendinx} holds by the well posedness of the forward problem.
Thus, we have
\begin{align}
\partial_u F^{(s)} \left( k_0, u^{(s)}(k_0) \right) = 
\mathcal S - k_0^2 \mathcal M + \imath k_0\mathcal B
\label{eqn:dufs}
\end{align}
and
\begin{align}
\partial_k F^{(s)} \left( k_0, u^{(s)}(k_0) \right) = 
(\mathcal - 2 k_0 \mathcal M +\imath \mathcal B) u^{(s)}(k_0).
\label{eqn:dkfs}
\end{align}
Plugging \eqref{eqn:dufs} and \eqref{eqn:dkfs} into \eqref{eqn:dkuk0}, we get
\begin{align}
\left[ \partial_k u^{(s)} \right](k_0) = 
\left( \mathcal S - k_0^2 \mathcal M + \imath k_0 \mathcal B \right)^{-1} 
\left( 2 k_0 \mathcal M u^{(s)} (k_0) + \imath \mathcal B u^{(s)}(k_0) \right).
\end{align}
Thus, in order to compute $w^{(s)} = \partial_k u^{(s)}(k_0)$ one has to solve 
\eqref{eqn:schrodderiv} with boundary condition \eqref{eqn:bcderiv}.

\section{Proof sketch of Theorem \ref{thm:fwdpde}}
\label{app:fwdpde}
Here we sketch the proof of Theorem \ref{thm:fwdpde}. 
We follow a setting similar to that in \cite{Wald2018,tataris2023reduced}. 
For that reason we study the complex conjugate version of \eqref{eqn:weaksom}, 
that is to find $u^{(s)}\in H^1(\Omega)$ such that for any $v \in H^1(\Omega)$ it satisfies
\begin{equation}
\int_\Omega {\nabla u^{(s)}} \cdot \overline {\nabla {v}} \; d \bx + 
\int_{\Omega} q {u^{(s)}}\overline{ v} d \bx - 
k^2 \int_\Omega {u^{(s)}} \overline{v }d \bx - 
\imath k \int_{\partial \Omega} {u^{(s)}} \overline{v} d\Sigma = 
\int_{\partial \Omega} {p_s} \overline{v} d\Sigma.
\label{eqn:weaksomconj}
\end{equation}

We begin by defining the forms 
$a_1,a_2:H^1(\Omega)^2\to \C$ as
\begin{align}
a_1(u,v) & =\int_\Omega \nabla u \cdot \cl {\nabla v} d\bx-\imath k\int_{\partial \Omega} u\cl vd\Sigma, \\
a_2(u,v)& =-\int_\Omega u\overline{v}d\bx +\frac{1}{k^2}\int_\Omega q u\overline{v}d\bx, 
\end{align}
for $u, v \in H^1(\Omega)$.
We note that $a_1$ is coercive and bounded, while $a_2$ is bounded. 

Let us define the linear Riesz isomorphism
\begin{align}
    \Phi: H^1(\Omega)\to \overline{H^1(\Omega)}', 
\end{align}
as $\Phi u= \langle u,\cdot \rangle_{\overline{H^1(\Omega)}'}$ for $u \in H^1(\Omega)$, and with $\overline{H^1(\Omega)}'$ denoting the antidual of the space $H^1(\Omega)$.
Since $a_1(u,\cdot)$ is an antilinear functional on $H^1(\Omega)$, 
and using the Lax-Milgram lemma, we define $\mathcal T:H^1(\Omega)\to H^1(\Omega)$ 
as
\begin{equation}
a_1(u,v) = \langle \mathcal Tu, v \rangle_{\overline{H^1(\Omega)}'},
\end{equation}
with $\mathcal T$ being one-to-one and onto. 

We also define the linear operator 
$\mathcal W: L^2(\Omega) \to \overline{H^1(\Omega)}'$, 
$u \stackrel{\mathcal W }{\mapsto} a_2(u,\cdot)$,
and two linear maps
\begin{align}
\mathcal A_1 & = \mathcal T^{-1} \Phi^{-1} \mathcal W: L^2(\Omega) \to H^1(\Omega),  \\
\mathcal A & = \mathcal A_1 \circ id_{H^1(\Omega) \to L^2(\Omega)}: 
H^1(\Omega) \stackrel{c}{\subspace} L^2(\Omega)\to H^1(\Omega) , \quad s \mapsto \mathcal A_1 s,
\end{align}
where $id_{H^1(\Omega) \to L^2(\Omega)}$ is the compact embedding operator of $H^1(\Omega)$ into $L^2(\Omega).$
Also, for $v \in H^1(\Omega)$ and $w \in H^1(\Omega)$, we have $a_1(\mathcal Av,w)=a_2(v,w).$
We claim that $\mathcal I+k^2 \mathcal A$ is one-to-one.

Finding a solution of \eqref{eqn:weaksomconj} is equivalent to finding $u^{(s)}\in H^1(\Omega)$ 
such that for all $v \in H^1(\Omega)$ it satisfies
\begin{align}
a_1(u^{(s)},v)+k^2 a_2(u^{(s)},v) & = \langle  P^{(s)},v\rangle_{\overline{H^1(\Omega)}'} \iff \\
a_1(u^{(s)},v)+k^2 a_1(\mathcal Au^{(s)}, v) & = \langle P^{(s)},v\rangle_{\overline{H^1(\Omega)}'} \iff \\
a_1(u^{(s)} + k^2\mathcal Au^{(s)}, v) & = \langle  P^{(s)},v\rangle_{\overline{H^1(\Omega)}'} \iff \\
\left\langle \mathcal T \left( u^{(s)} + k^2 \mathcal Au^{(s)} \right), v \right\rangle_{\overline{H^1(\Omega)}'} & 
= \langle P^{(s)}, v\rangle_{\overline{H^1(\Omega)}'},
\end{align}
hence
\begin{align}
\Phi \mathcal T (\mathcal I+k^2\mathcal A)u^{(s)} & \stackrel{\overline{H^1(\Omega)}'}{=}P^{(s)} \iff  \\
(\mathcal I+k^2 \mathcal A) u^{(s)} & = \mathcal T^{-1} \Phi^{-1} P^{(s)} \in H^1(\Omega).
\label{eqn:operus}
\end{align}
Since $\mathcal A \in \mathcal{L}(H^1(\Omega),H^1(\Omega))$ is compact and $\mathcal I+k^2\mathcal A$ is injective, using the Fredholm alternative we obtain that there exists a unique element $u^{(s)}\in H^1(\Omega) $ that satisfies \eqref{eqn:operus}. We also obtain the forward stability estimate
\begin{align*}
    \| u^{(s)} \|_{H^1(\Omega)} \leq & 
    \| (\mathcal I+k^2\mathcal A)^{-1}\|_{\mathcal{L}(H^1(\Omega),H^1(\Omega))}  
    \| \mathcal T^{-1}\|_{\mathcal{L}(H^1(\Omega),H^1(\Omega))} \\
    & \| \Phi^{-1}\|_{\mathcal{L} (H^1(\Omega),\overline{H^1(\Omega)}')} 
    \| P^{(s)} \|_{\overline{H^1(\Omega)}'}.
\end{align*}
The above argument establishes that for all $v \in H^1(\Omega)$, a unique weak solution of 
\eqref{eqn:weaksomconj} exists. Therefore, by taking the complex conjugate of 
\eqref{eqn:weaksomconj}, we conclude that there exists a unique solution of \eqref{eqn:weaksom}.

\bibliographystyle{siamplain}
\bibliography{bibliography}
\end{document}